\documentclass{amsart}

\usepackage{enumerate, pdfsync, booktabs, bm} 

\usepackage{tikz}
\usetikzlibrary{matrix,decorations.pathmorphing}
\usetikzlibrary{arrows,shapes,snakes,automata,backgrounds,petri}

\usepackage[initials]{amsrefs}
\usepackage{mathrsfs} % Use \mathscr

\theoremstyle{plain}
\newtheorem{theorem}{Theorem}[section]
\newtheorem{corollary}[theorem]{Corollary}
\newtheorem{proposition}[theorem]{Proposition}
\newtheorem{lemma}[theorem]{Lemma}

\theoremstyle{definition}

\newtheorem*{definition*}{Definition}
\newtheorem*{example*}{Example}
\newtheorem*{remark*}{Remark}

\newcommand{\QQQ}{\mathscr{Q}}
\newcommand{\UUU}{\mathscr{U}}

\begin{document}
\title{Number theoretical properties of Romik's dynamical system}
%\title{Lagrange Theorem for Romik system}
\author{Byungchul Cha}
\address{Muhlenberg College, 2400 Chew st, Allentown, PA, 18104, USA}
\email{cha@muhlenberg.edu}

\author{Dong Han Kim}
\address{Department of Mathematics Education, Dongguk University - Seoul, 30 Pildong-ro 1-gil, Jung-gu, Seoul, 04620 Korea}
\email{kim2010@dongguk.edu}

\thanks{Research supported by the National Research Foundation of Korea (NRF-2018R1A2B6001624).}



\begin{abstract}
We study a dynamical system that was originally defined by Romik in 2008 using an old theorem of Berggren concerning Pythagorean triples.
Romik's system is closely related to the Farey map on the unit interval which generates an additive continued fraction algorithm.
We explore some number theoretical properties of the Romik system.
In particular, we prove an analogue of Lagrange's theorem in the case of the Romik system on the unit quarter circle, which states that a point possesses an eventually periodic digit expansion if and only if the point is defined over a real quadratic extension field of rationals.
\end{abstract}

\maketitle

\section{Introduction}

One of the oldest and most classical theorems in number theory is perhaps the infinitude of \emph{primitive Pythagorean triples}, that is, positive integer triples $(x, y, z)$ without common factor satisfying $x^2 + y^2 = z^2$. 
Less known is the fact that the set of all such triples can be equipped with a certain tree-like structure. 
To explain, let
\begin{equation}\label{Ms}
M_1=
\begin{pmatrix}
-1 & 2 & 2 \\
-2 & 1 & 2 \\
-2 & 2 & 3 \\
\end{pmatrix},
\quad
M_2=
\begin{pmatrix}
1 & 2 & 2 \\
2 & 1 & 2 \\
2 & 2 & 3 \\
\end{pmatrix},
\quad
M_3=
\begin{pmatrix}
1 & -2 & 2 \\
2 & -1 & 2 \\
2 & -2 & 3 \\
\end{pmatrix}.
\end{equation}
\begin{figure}
\begin{center}
\tikzset{triarrow/.pic={
    \draw (0, 0) -- (0, 0.5) node[above]{$\vdots$};
    \draw (-0.1, 0) -- (-0.2, 0.5) node[above]{$\vdots$};
    \draw (0.1, 0) -- (0.2, 0.5) node[above]{$\vdots$};
  }
}
\begin{tikzpicture}[->, >=latex', auto, xscale=1.2]
\node (base) at (0, -0.5) {$(3, 4, 5)$};

\node (1f1) at (-3.5, 1) {$(15, 8, 17)$};
\node (1f2) at (0, 1) {$(21, 20, 29)$};
\node (1f3) at (3.5, 1) {$(5, 12, 13)$};

\pic at (-4.5, 3.3) {triarrow};
\pic at (-3.5, 3.8) {triarrow};
\pic at (-2.5, 4.3) {triarrow};

\node (2f1) at (-4.5, 3) {$(35, 12, 37)$};
\node (2f2) at (-3.5, 3.5) {$(65, 72, 97)$};
\node (2f3) at (-2.5, 4) {$(33, 56, 65)$};

\node (2f4) at (-1, 3) {$(77, 36, 85)$};
\node (2f5) at (0, 3.5) {$(119, 120, 169)$};
\node (2f6) at (1.5, 4) {$(39, 80, 89)$};

\pic at (-1.3, 3.3) {triarrow};
\pic at (0, 3.8) {triarrow};
\pic at (1.5, 4.3) {triarrow};

\node (2f7) at (2.5, 3) {$(45, 28, 53)$};
\node (2f8) at (3.5, 3.5) {$(55, 48, 73)$};
\node (2f9) at (4.5, 4) {$(7, 24, 25)$};

\pic at (2.5, 3.3) {triarrow};
\pic at (3.5, 3.8) {triarrow};
\pic at (4.5, 4.3) {triarrow};

\draw (base) to node[font=\footnotesize]{${M_1}$} (1f1);
\draw (base) to node[font=\footnotesize]{${M_2}$} (1f2);
\draw (base) to node[below right, font=\footnotesize]{${M_3}$} (1f3);

\draw (1f1) to node[font=\footnotesize]{${M_1}$} (2f1); 
\draw (1f1) to node[font=\footnotesize]{${M_2}$} (2f2); 
\draw (1f1) to node[right, font=\footnotesize]{${M_3}$} (2f3); 

\draw (1f2) to node[font=\footnotesize]{${M_1}$} (2f4); 
\draw (1f2) to node[font=\footnotesize]{${M_2}$} (2f5); 
\draw (1f2) to node[right, font=\footnotesize]{${M_3}$} (2f6); 

\draw (1f3) to node[font=\footnotesize]{${M_1}$} (2f7); 
\draw (1f3) to node[font=\footnotesize]{${M_2}$} (2f8); 
\draw (1f3) to node[right, font=\footnotesize]{${M_3}$} (2f9); 
\end{tikzpicture}
\begin{tikzpicture}[->, >=latex', auto, xscale=1.2]
%\draw[help lines,step=.5] (-5,-1) grid (5,5);
\node (base) at (0, -0.5) {$(4, 3, 5)$};

\node (1f1) at (-3.5, 1) {$(12, 5, 13)$};
\node (1f2) at (0, 1) {$(20, 21, 29)$};
\node (1f3) at (3.5, 1) {$(8, 15, 17)$};

\node (2f1) at (-4.5, 3) {$(24, 7, 25)$};
\node (2f2) at (-3.5, 3.5) {$(48, 55, 73)$};
\node (2f3) at (-2.5, 4) {$(28, 45, 53)$};
%\node (2f3) at (-2.5, 4) {$(24, 7, 25)$};

\pic at (-4.5, 3.3) {triarrow};
\pic at (-3.5, 3.8) {triarrow};
\pic at (-2.5, 4.3) {triarrow};

\node (2f4) at (-1, 3) {$(80, 39, 89)$};
\node (2f5) at (0, 3.5) {$(120, 119, 169)$};
\node (2f6) at (1.5, 4) {$(36, 77, 85)$};
%\node (2f6) at (1.5, 4) {$(80, 39, 89)$};

\pic at (-1.3, 3.3) {triarrow};
\pic at (0, 3.8) {triarrow};
\pic at (1.5, 4.3) {triarrow};

\node (2f7) at (2.5, 3) {$(56, 33, 65)$};
\node (2f8) at (3.5, 3.5) {$(72, 65, 97)$};
\node (2f9) at (4.5, 4) {$(12, 35, 37)$};
%\node (2f9) at (4.5, 4) {$(56, 33, 65)$};

\pic at (2.5, 3.3) {triarrow};
\pic at (3.5, 3.8) {triarrow};
\pic at (4.5, 4.3) {triarrow};

\draw (base) to node[font=\footnotesize]{${M_1}$} (1f1);
\draw (base) to node[font=\footnotesize]{${M_2}$} (1f2);
\draw (base) to node[below right, font=\footnotesize]{${M_3}$} (1f3);

\draw (1f1) to node[font=\footnotesize]{${M_1}$} (2f1); 
\draw (1f1) to node[font=\footnotesize]{${M_2}$} (2f2); 
\draw (1f1) to node[right, font=\footnotesize]{${M_3}$} (2f3); 

\draw (1f2) to node[font=\footnotesize]{${M_1}$} (2f4); 
\draw (1f2) to node[font=\footnotesize]{${M_2}$} (2f5); 
\draw (1f2) to node[right, font=\footnotesize]{${M_3}$} (2f6); 

\draw (1f3) to node[font=\footnotesize]{${M_1}$} (2f7); 
\draw (1f3) to node[font=\footnotesize]{${M_2}$} (2f8); 
\draw (1f3) to node[right, font=\footnotesize]{${M_3}$} (2f9); 
\end{tikzpicture}
\end{center}

\caption{Berggren trees for Pythagorean triples\label{PythagoreanTree}} 
\end{figure}
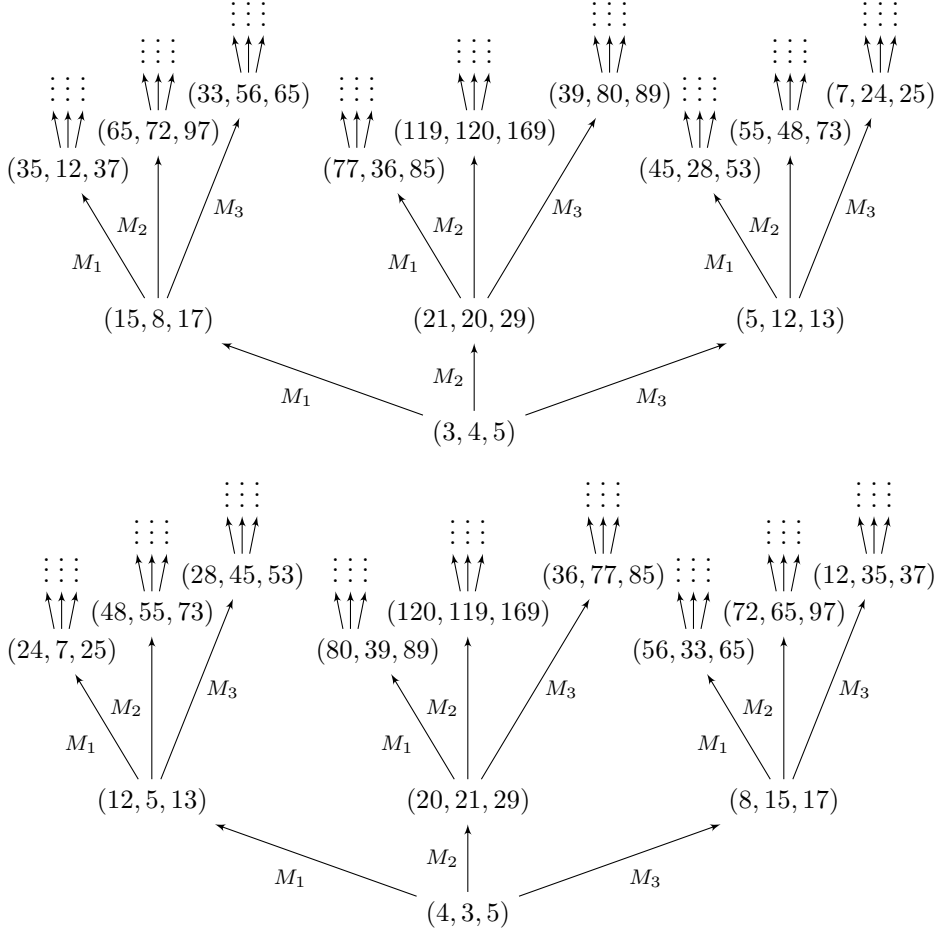
Then the repeated actions of $M_1, M_2, M_3$ via left-multiplication on (the column vectors) $(3, 4, 5)$ and $(4, 3, 5)$ generate all primitive Pythagorean triples  
and each primitive Pythagorean triple shows up in the trees exactly once, as pictured in Figure~\ref{PythagoreanTree}.
As far as we are aware, the oldest literature containing this theorem is the paper \cite{Ber34} by Berggren. 
We will refer the trees in Figure~\ref{PythagoreanTree} as the \emph{Berggren trees}.
See \cite{Alp05} and \cite{Bar63} for the proofs of this theorem and related discussion. 
Also, the paper \cite{Rom08} by Romik contains an extensive list of bibliography on Berggren's theorem.

In the same paper, Romik initiates the investigation of a dynamical system $T: \QQQ \longrightarrow \QQQ$, where $\QQQ$ is the (closed) unit quarter circle
\[
\QQQ = \{ (x, y)\in \mathbb{R}^2 \mid x\ge0, y\ge0 \text{ and }x^2 + y^2 = 1 \},
\]
and $T(x, y)$ is defined to be
\[
T(x, y) = \left(
\frac{|2 - x - 2y|}{3 - 2x - 2y},
\frac{|2 - 2x - y|}{3 - 2x - 2y}
\right).
\]
%In fact, Romik defines $\QQQ$ to be an \emph{open} arc, not a closed one. 
%Our definition is a slight modification of Romik's, which doesn't result in any substantial difference.
This dynamical system naturally arises from the Berggren trees in Figure~\ref{PythagoreanTree}. In fact, $T$ is the unique continuous map from $\QQQ$ to itself, which is ``parent-finding'' in the following sense. 
If $(\frac ac, \frac bc)$ is a rational point on $\QQQ$ represented by a primitive Pythagorean triple $(a, b, c)$, then $T(\frac ac, \frac bc) = (\frac{a'}{c'}, \frac{b'}{c'})$, finding the \emph{parent} Pythagorean triple $(a', b', c')$ of $(a, b, c)$ in the Berggren trees.
\begin{figure}
	\begin{center}
		\begin{tikzpicture}[scale=0.8]
			\node[below left] at (0, 0) {$O$};
			\draw[->] (-0.5, 0) -- (5.5, 0) node[right] {$x$};
			\draw[->] (0, -0.5) -- (0, 5.5) node[above] {$y$};
			\draw[very thick] (5, 0) node [below] {$1$}
			arc (0:90:5) node[left] {$1$};

			\draw[fill] (3, 4) circle (0.08) node[below left] {$(\frac35, \frac45)$};
			\draw[fill] (4, 3) circle (0.08) node[below left] {$(\frac45, \frac35)$};
			\draw[<->, dotted, thick] (5.2, 0) 
			arc (0: 36.8:5.2) node[midway, above right] {$d=1$};
			\draw[<->, dotted, thick] ([shift=(36.8:5.2)]0, 0) 
			arc (36.8: 53.1: 5.2) node[midway, above right] {$d=2$};
			\draw[<->, dotted, thick] ([shift=(53.1:5.2)]0, 0) 
			arc (53.1: 90: 5.2) node[midway, above right] {$d=3$} ;
		\end{tikzpicture}
	\caption{Digits of the points on $\QQQ$ \label{DigitPicture}}
	\end{center}
	
\end{figure}
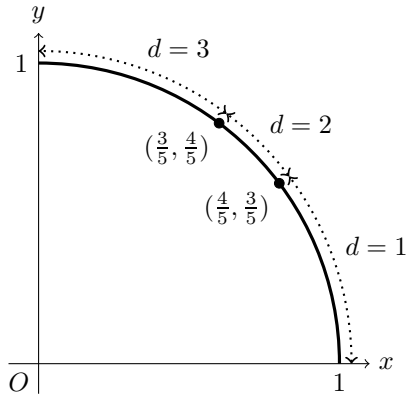
Additionally, Romik assigns a digit $d = d(x, y) \in \{ 1, 2, 3 \}$ for each $(x, y) \in \QQQ$ according to
\begin{equation}\label{DigitDefinition}
d(x, y) = \begin{cases}
1 & \text{ if } \frac45 < x \le 1, \\
2 & \text{ if } \frac35 < x < \frac45, \\
3 & \text{ if } 0\le x < \frac35, \\
\end{cases}
\end{equation}
as in Figure~\ref{DigitPicture}. 
Then the digit expansion of $(x, y)$ is defined to be
\begin{equation}\label{RomikExpansion}
	(x, y) = [d_1, d_2, \dots ]_{\QQQ},
\end{equation}
where $d_j = d(T^{j-1}(x, y))$ for $j=1, 2, \dots.$ We call this the \emph{Romik digit expansion} of $(x, y)$. 
For the ``boundary cases'' $x=\frac35$ and $x =\frac45$, 
we allow both adjacent digits to be valid, instead of using ``oe'' and ``eo'' as in \cite{Rom08}. 
Namely,
\[
d(\tfrac45,\tfrac35) =1 \text{ or } 2, \qquad
d(\tfrac35,\tfrac45) =2 \text{ or } 3. 
\]
Note that $T(\tfrac35,\tfrac45)=(1, 0)$ and
$T(\tfrac45,\tfrac35)=(0, 1)$
are the only fixed points of $T(x, y)$, and their Romik digit expansions are
\[
	(1, 0) = [1^{\infty}]_{\QQQ}:= [1, 1, \dots ]_{\QQQ},
\text{ and }
	(0, 1) = [3^{\infty}]_{\QQQ}:= [3, 3, \dots ]_{\QQQ}.
\]
As a consequence, Theorem 2 in \cite{Rom08} says that a point $(x, y) \in \QQQ$ is rational if and only if
its Romik digit expansion ends with either $1^{\infty}$ or $3^{\infty}$. 
Moreover, 
except for $(1, 0)$ and $(0, 1)$, 
every rational point $(x, y)$ has \emph{two} endings in its Romik digit expansions
\[
	(x, y) = [d_1, d_2, \dots, 2, 1^{\infty}]_{\QQQ} \text{ and }  [d_1, d_2, \dots, 3, 1^{\infty}]_{\QQQ},
\]
or
\[
	(x, y) = [d_1, d_2, \dots, 2,  3^{\infty}]_{\QQQ} \text{ and }  [d_1, d_2, \dots, 1, 3^{\infty}]_{\QQQ}.
\]

The Romik digit expansion of a point on $\QQQ$ resembles classical (simple) continued fraction expansions of real numbers. 
In fact, as Romik explains in \cite{Rom08}, the dynamical system $(\QQQ, T)$ is conjugate to
a dynamical system on the unit interval $[0, 1]$ constructed from, what he calls, a modified Euclidean algorithm. 

Does $(\QQQ, T)$ exhibit any number theoretical property shared by the dynamical system associated with the classical continued fraction? 
For example, Romik asks in \cite{Rom08} if an analogue of Lagrange's theorem for $(\QQQ, T)$ is true. 
That is, is it true that a point $(x, y)\in\QQQ$ has an eventually periodic Romik digit expansion if and only if $(x, y)$ is defined over a (real) quadratic extension of $\mathbb{Q}$?
The main result in our paper answers this question affirmatively. 
Our proof leverages the structure of a \emph{quadratic space} $(\mathbb{R}^3, Q(\mathbf{x}))$ where $Q(\mathbf{x}) = x_1^2 + x_2^2 - x_3^2$ is the Pythagorean quadratic form. In this sense, the present paper is a natural continuation of the work \cite{CNT} by Cha, Nguyen and Tauber. 

Perhaps not surprisingly, our proof of the Lagrange theorem is reminiscent of Lagrange's original proof of his theorem for the usual continued fraction expansion.
The heart of Lagrange's classical proof is to show that the discriminants of defining equations of the irrationals with the same tails in their continued fraction expansions lie in a bounded subset of $\mathbb{R}$.
In our case, a key step is to show that the \emph{$Q$-cross products} (see \S\ref{SectionTwistedProduct}) arising from the vectors in the same Romik $T$-orbit lie in a bounded subset of $\mathbb{R}^3$.

We point out here that one can deduce Lagrange's theorem for the Romik system as a corollary from another general theorem, namely, Panti's theorem in \cite{Pan09}. 
Panti proves a version of Lagrange's theorem which is applicable to many systems arising from unimodular partitions of the unit interval. 
The aforementioned dynamical system on the unit interval to which $(\QQQ, T)$ is conjugate is an example of such a system. 
Thus, Panti's theorem in \cite{Pan09} can be applied in this context to prove the Lagrange theorem for $(\QQQ, T)$.  %\footnote{
After a version of this paper had been completed, we recently learned from \cite{Pan19} that Panti proved in the Lagrange and Galois theorems for, what he calls, all billiard maps based on unimodular partitions. 
%After we finished writing the paper, we found Panti also recently proved the Lagrange theorem and the Galois theorem for all billiard maps based on unimodular partitions in \cite{Pan19}.%}

Even though Panti's theorem in \cite{Pan09} gives a simpler proof for the Lagrange theorem for $(\QQQ, T)$ it seems to us that the tools we develop in the present paper are well-suited to the study of other number theoretical properties regarding $(\QQQ, T)$. 
In particular, our set-up for the Romik system can provide flexible tools in the study of \emph{intrinsic Diophantine approximation}, following Kleinbock, Merril, and their collaborators in \cite{KM15} and \cite{FKMS}. 
The authors intend to pursue this in a follow-up work in the near future.
%As an example, our analogue (Corollary~\ref{Galois}) of Galois' theorem, which relates the continued fraction expansion of a quadratic irrational and that of its Galois conjugate, is 
%More generally, we believe Romik digits and other methods in the present paper can be used to approximate irrational points on $\QQQ$ by the rational ones. 

The rest of this paper is organized as follows. In \S\ref{SecReview}, we set up notations and recall some background materials on quadratic spaces, as well as prior results mainly from \cite{Con} and \cite{CNT}. 
We review the geometric construction from \cite{CNT} and provide a self-contained proof of Berggren's theorem, which is essentially due to Conrad \cite{Con}. 
There are already many proofs available in the literature for Berggren's theorem. 
However, our presentation provides a convenient starting point for our discussion later.
%The generalizations and applications of this construction for other quadratic forms than the Pythagorean one are the main results in \cite{CNT}.
The main result of the paper is contained in \S\ref{SecProof}, where we complete the proofs of Lagrange's and Galois' theorems for the Romik system. 
Some open questions and potential future developments are described in \S\ref{SecOpen}.

\section{Background materials and review on the Romik system}\label{SecReview}
\subsection{Notational convention and quadratic spaces}\label{SecNotation}
We will use bold-faced letters, such as $\mathbf{v}, \mathbf{w}$, etc., to denote vectors in $\mathbb{R}^3$. They are regarded as column vectors, so that a $3\times 3$ matrix acts on them via left-multiplication. 
Points on $\mathbb{R}^2$ are usually denoted by capital letters, such as $P$.
We write $x_1, x_2, x_3$ for the standard coordinate functions in $\mathbb{R}^3$ and we use $x$ and $y$ for $\mathbb{R}^2$. 
Define the unit circle
\[
	\UUU = \{ (x, y) \in \mathbb{R}^2 \mid x^2 + y^2 = 1 \} 
\]
and the quarter circle
\[
	\QQQ = \{ (x, y) \in \mathbb{R}^2 \mid x, y\ge 0 \text{ and }x^2 + y^2 = 1  \}.
\]
Also, we define a projection map $\pi$ to be
\begin{equation}\label{ProjectionMap}
	\pi: \mathbb{R}^3 - \{ x_3 = 0 \} \longrightarrow \mathbb{R}^2, \qquad
(x_1, x_2, x_3) \mapsto (x, y) := (x_1/x_3, x_2/x_3).
\end{equation}
We will say that $\mathbf{v}$ \emph{represents} $P$ when $\pi(\mathbf{v}) = P$.

Let 
\[
	Q(\mathbf{x}) = x_1^2 + x_2^2 - x_3^2,
\]
which we will refer as the Pythagorean quadratic form on $\mathbb{R}^3$. 
However, many of our results here can be applied with some minimal modifications to other quadratic forms $Q(\mathbf{x})$ satisfying certain technical conditions,
the conditions ($Q$-I)---($Q$-III) in \cite{CNT}, to be precise. 
With such other cases in mind, we will try to present our arguments in the form that can be generalized easily in the future.
Recall that a \emph{quadratic space} ($\mathbb{R}^3, Q(\mathbf{x})$) is a vector space $\mathbb{R}^3$ equipped with a quadratic form $Q(\mathbf{x})$ on $\mathbb{R}^3$. 
It is well-known that the set of all quadratic forms $Q(\mathbf{x})$ is in one-to-one correspondence with the set of all symmetric bilinear forms on $\mathbb{R}^3$, as well as that of all $3\times 3$ symmetric matrices with real coefficients.
In general, the relationship among an arbitrary $Q(\mathbf{x})$, the corresponding bilinear form, and $M_Q$ is given by
\begin{equation}\label{TripleCorrespondence}
	\langle \mathbf{x}, \mathbf{y} \rangle = 
	\frac12(Q(\mathbf{x} + \mathbf{y}) - Q(\mathbf{x}) - Q(\mathbf{y})) 
	=
	\mathbf{x}^TM_Q\mathbf{y},
\end{equation}
for all $\mathbf{x}, \mathbf{y}$ in $\mathbb{R}^3$.

In our case of the Pythagorean form, the corresponding bilinear form is
\[ 
\langle \mathbf{x}, \mathbf{y} \rangle
=
x_1y_1 + x_2y_2 - x_3y_3,
\]
where $\mathbf{x} = (x_1, x_2, x_3)$ and $\mathbf{y} = (y_1, y_2, y_3)$, and
\[
M_Q 
=
\begin{pmatrix} 1 & 0 & 0 \\ 0 & 1 & 0 \\ 0 & 0 & -1\\
\end{pmatrix}.
\]

Define 
\[
C_Q = \{ \mathbf{x}\in\mathbb{R}^3 \mid Q(\mathbf{x}) = 0 \},
\]
to be the set of all $Q$-null vectors in $\mathbb{R}^3$. 
Also, if a linear map $A$ on $\mathbb{R}^3$ to itself preserves $Q(\mathbf{x})$, that is, $Q(A\mathbf{x}) = Q(\mathbf{x})$ for all $\mathbf{x}\in \mathbb{R}^3$, then we say that $A$ is \emph{orthogonal with respect to $Q(\mathbf{x})$}. 
The group of all orthogonal maps is denoted by $O_Q(\mathbb{R})$. Clearly, every $A\in O_Q(\mathbb{R})$ maps $C_Q$ to itself.  
%In addition, if $A\in O_Q(\mathbb{Z})$, that is, if $A\mathbf{x} \in \mathbb{Z}^3$ for every $\mathbf{x}\in\mathbb{Z}^3 \subset \mathbb{R}^3$, then 
%the entries of $A(a, b, c)$ have no common factor,
%whenever $a, b, c$ are integers without common factor.

\subsection{Reflections and their actions on $\UUU$ and $C_Q$}\label{SecReflection}
Fix $\mathbf{z}\in\mathbb{R}^3$ with $Q(\mathbf{z})\neq 0$. 
The \emph{reflection $s_{\mathbf{z}}:\mathbb{R}^3 \longrightarrow \mathbb{R}^3$ of} $\mathbf{z}$ is defined to be
\begin{equation}\label{Reflection}
s_{\mathbf{z}}(\mathbf{x}) = 
\mathbf{x} -
2
\frac{\langle \mathbf{x}, \mathbf{z} \rangle}{Q(\mathbf{z})} \mathbf{z}.
\end{equation}
It is easy to prove that $s_{\mathbf{z}} \in O_Q(\mathbb{R})$.
Also, $s_{\mathbf{z}}^2 = \mathbf{1}_{\mathbb{R}^3}$, the identity map on $\mathbb{R}^3$, so that $s_{\mathbf{z}}$ is its own inverse. 
In addition, $\det(s_{\mathbf{z}}) = -1$ because $\mathbf{z}$ is an eigenvector of $s_{\mathbf{z}}$ with the eigenvalue $-1$, while 
\[
	\{ \mathbf{x} \in \mathbb{R}^3 \mid
	\langle \mathbf{x}, \mathbf{z} \rangle = 0 \}
\]
is a two-dimensional eigenspace with the eigenvalue $1$. 

Define $U_1, U_2, U_3$ by
\begin{equation}\label{DefinitionUj}
U_1 = \begin{pmatrix}
1 & 0 & 0 \\
0 & -1 & 0 \\
0 & 0 & 1 \\
\end{pmatrix},
\quad
U_2 = \begin{pmatrix}
-1 & 0 & 0 \\
0 & -1 & 0 \\
0 & 0 & 1 \\
\end{pmatrix},
\quad
U_3 = \begin{pmatrix}
-1 & 0 & 0 \\
0 & 1 & 0 \\
0 & 0 & 1 \\
\end{pmatrix}.
\end{equation}
Then it is easy to show that $U_1=s_{(0, 1, 0)}$ and $U_3=s_{(1, 0, 0)}$ (with the matrices being understood via the standard basis of $\mathbb{R}^3$.)
Also, $U_2 = U_1U_3$ is the composite of the two reflections $s_{(0,1, 0)}$ and $s_{(1, 0, 0)}$.
Next, let 
\begin{equation}\label{DefinitionH}
H = 
s_{(1, 1, 1)}
=
\begin{pmatrix}
-1 & -2 & 2 \\
-2 & -1 & 2 \\
-2 & -2 & 3 \\
\end{pmatrix}.
\end{equation}
Then, the matrices $M_1, M_2, M_3$ in \eqref{Ms} are 
\begin{equation}\label{MandHU}
M_1 = HU_1, \quad
M_2 = HU_2, \quad
\text{ and }
\quad
M_3 = HU_3,
\end{equation}
as first noted by Conrad in \cite{Con} and by Berggren \cite{Ber34} in a more indirect way.
Equivalently,
\begin{equation}\label{MInverseandHU}
	M_1^{-1} = U_1H, \quad
	M_2 ^{-1}= U_2H \quad
\text{ and }
\quad
M_3 ^{-1}= U_3H.
\end{equation}

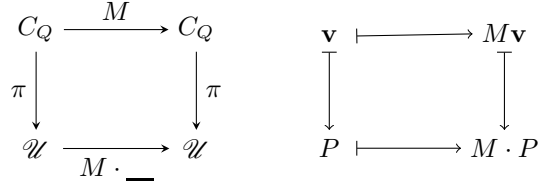
\begin{figure}
\begin{center}
\begin{tikzpicture}
  \matrix (m) [matrix of math nodes,row sep=3em,column sep=4em,minimum width=2em]
  {
     C_Q & C_Q \\
     \UUU & \UUU \\};
  \path[-stealth]
    (m-1-1) edge node [left] {$\pi$} (m-2-1)
            edge node [above] {$M$} (m-1-2)
    (m-2-1.east|-m-2-2) edge node [below] {$M\cdot \underline{\hspace{1em}}$}
            node [above] {$$} (m-2-2)
    (m-1-2) edge node [right] {$\pi$} (m-2-2);
\end{tikzpicture}
\qquad
\begin{tikzpicture}
  \matrix (m) [matrix of math nodes,row sep=3em,column sep=4em,minimum width=2em]
  {
     \mathbf{v} & M\mathbf{v} \\
     P & M\cdot P \\};
  \path[-stealth]
    (m-1-1) edge [|->]  (m-2-1)
            edge [|->] (m-1-2)
    (m-2-1.east|-m-2-2) edge [|->] node [below] {\phantom{$M\cdot \underline{\hspace{1em}}$}} (m-2-2)
    (m-1-2) edge [|->] (m-2-2);
\end{tikzpicture}
\end{center}
\caption{Action of $M$ on $C_Q$ and $\UUU$ \label{ActionofM}}
\end{figure}
Central to us is to describe geometrically the actions of $U_1, U_2, U_3, H$ on $C_Q$, 
as well as their induced actions on $\UUU$, whose meaning we now make precise below.
Suppose that $M$ is any one of $U_1, U_2, U_3,$ and $H$, or a finite product of them. 
If $P$ is in $\UUU$ then we denote by $M\cdot P$ the point represented by $M\mathbf{v}$ when $\mathbf{v}$ is any vector representing $P$.
This is summarized in Figure~\ref{ActionofM}.

Apply this for $M = U_j$. Then we obtain from \eqref{DefinitionUj} that
\begin{align*}
	U_1\cdot(x, y) &= (x, -y), \\
	U_2\cdot(x, y) &= (-x, -y), \\
	U_3\cdot(x, y) &= (-x, y), \\
\end{align*}
for any $(x, y) \in \UUU$.
As for $M=H$, we compute $H\mathbf{v}$ using \eqref{DefinitionH} for a few $\mathbf{v}$'s. The results are summarized in the table below.
\[
	\begin{array}{ll}
	\toprule
		\mathbf{v} & H\mathbf{v} \\
	\midrule
	(0, 1, 1) & (0, 1, 1)  \\
	(3, 4, 5) & (-1, 0, 1)  \\
	(4, 3, 5) & (0, -1, 1)  \\
	(1, 0, 1) & (1, 0, 1)  \\
	\bottomrule
	\end{array}
\]
\begin{figure}
\begin{center}
	\begin{tikzpicture}[scale=2.5, >=latex', auto]
%\begin{tikzpicture}[->, >=latex', auto, xscale=1.2]
		\draw[->](-1.2, 0) -- (1.5, 0) node[right]{$x$};
		\draw[->](0, -1.2) -- (0, 1.5) node[above]{$y$};
		\draw (0, 0) circle (1);
		\draw[ultra thick] (1, 0) arc (0:90:1);

		\draw[fill] (0.8, 0.6) circle (0.02) node[above right] {$(\frac45, \frac35)$};
		\draw[fill] (0.6, 0.8) circle (0.02) node[above right] {$(\frac35, \frac45)$};

		\draw[fill] (0, 1) circle (0.02) node[above right] {$1$};
		\draw[fill] (1, 0) circle (0.02) node[above right] {$1$};

		\draw[fill] (0, -1) circle (0.02) node[below left ] {$-1$};
		\draw[fill] (-1, 0) circle (0.02) node[below left ] {$-1$};

		\draw[<->, dashed] (0.8, 0.6) -- (0, -1);
		\draw[<->, dashed] (0.6, 0.8) -- (-1, 0);

		\draw[dotted] (-0.3, 1.3) -- (1.3, -0.3) node[below right]{$\ell_1$};
		\draw[dotted] (0.5, 1.5) -- (1.5, 0.5) node[below right]{$\ell_2$};
\end{tikzpicture}
\end{center}
\caption{Action of $H$ on $\UUU$ \label{ActionofH}}
\end{figure}
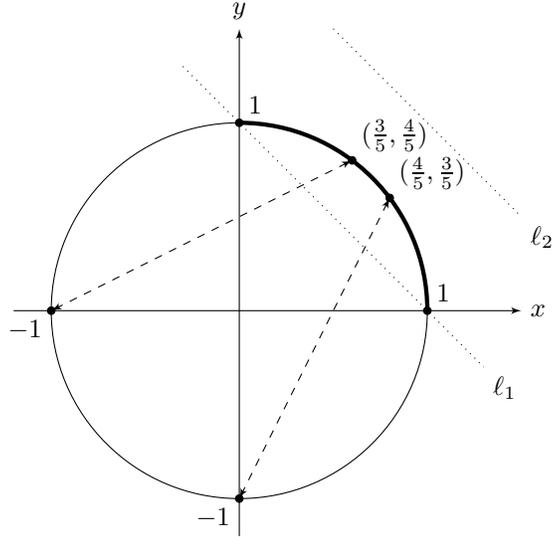
Remembering that $H$ is a reflection, thus of order 2, 
we see that $H$ sends the subarc of $\QQQ$ between $(1, 0)$ and $(\frac45, \frac35)$ onto the quarter circle in the fourth quadrant, and vice versa. 
Recall from the definition of $d(x, y)$ in \eqref{DigitDefinition} and in Figure~\ref{DigitPicture} that
this arc is precisely the set of $(x, y)$ whose Romik digit $d(x, y) = 1$. 
Likewise, $H$ moves back and forth the subarc of $\QQQ$ with $d(x, y) = 2$ onto the quarter circle in the third quadrant.
Similarly, the subarc with $d(x, y) = 3$ is mapped by $H$ onto the quarter circle in the second quadrant.
This gives a geometric description of the actions of $H$ and $U_j$'s on $\UUU$, as is pictured in Figure~\ref{ActionofH}.

We now consider the actions of $H$ and $U_j$'s on $C_Q$. 
As we understand their actions on $\UUU$ it is enough to know the $x_3$-coordinates of $H\mathbf{v}$ and $U_j \mathbf{v}$. 
First of all, $U_j$ obviously preserves the $x_3$-coordinates. 
As for $H$, we use the following proposition.
%we let $\mathbf{v} = (a, b, c) \in C_Q$ and write $(a', b', c') = H\mathbf{v}$. 
\begin{proposition}\label{EmilyLemma}
	Let $\mathbf{v} = (a, b, c)$ (not necessarily in $C_Q$) with $c\neq0$ and write $H\mathbf{v}=(a', b', c')$. Then
	\begin{equation}\label{ConditionDecrease}
       1< \frac ac + \frac bc <2,
	\end{equation}
	if and only if $|c'| < |c|$.
\end{proposition}
\begin{proof}
A straightforward calculation shows $c'=-2a - 2b +3c$. 
The proof then follows immediately from the identity
\begin{align*}
	c'^2 - c^2  &= (-2a -2b +3c)^2 - c^2\\ % = 4(c - (a + b))(2c - (a + b)) \\
       &= 4c^2 \left( \frac ac + \frac bc - 1\right)\left( \frac ac + \frac bc - 2 \right).
\end{align*}
\end{proof}
Note that the condition \eqref{ConditionDecrease} is equivalent to saying that $\mathbf{v}$ represents a point $(x, y)$ in $\mathbb{R}^2$
between the two lines
\[
	\ell_1: x + y = 1 \quad \text{ and }\quad \ell_2:x + y = 2,
\]
as depicted in Figure~\ref{ActionofH}.

Now, suppose $\mathbf{v}=(a, b, c)\in C_Q$ and write $H\mathbf{v} = (a', b', c')$. 
If $\mathbf{v}$ represents a point $P$ in (the interior of) $\QQQ$, then $\mathbf{v}$ satisfies the condition \eqref{ConditionDecrease}, therefore,
$|c'| < |c|$.
In other words, $H$ brings $\mathbf{v}$ closer to the origin if $P\in \QQQ$.
Likewise, while $H$ moves $\mathbf{v}$ away from the origin if $P\in \UUU-\QQQ$.

As a final note, we claim here that, if $c > 0$ for any $\mathbf{v} = (a, b, c) \in C_Q$, then $c'>0$ as well.
To prove this, define
\[
	C_Q^+ = \{ (a, b, c) \in C_Q \mid c > 0 \}
\]
and 
\[
	C_Q^- = \{ (a, b, c) \in C_Q \mid c < 0 \}.
\]
Note that $C_Q^+$ and $C_Q^-$ are disjoint and connected, and that $H$ is a continuous invertible map leaving stable the set
\[
	C_Q - \{\mathbf{0}\} = C_Q^+ \cup C_Q^-.
\]
So, $H$ must leave stable each of $C_Q^+$ and $C_Q^-$ separately, or $H$ must swap them.
But, the second is impossible because $H$ fixes $(1, 0, 1) \in C_Q^+$.
So, $H$ must leave stable each of $C_Q^{\pm}$ separately.

\begin{figure}
\begin{center}
	\begin{tikzpicture}[scale=2, >=latex', auto]
		\draw[->](-1.2, 0) -- (1.2, 0) node[right]{$x$};
		\draw[->](0, -1.2) -- (0, 1.2) node[above]{$y$};
		\draw (0, 0) circle (1);
		\draw[ultra thick] (1, 0) arc (0:90:1);

		\draw[fill] (0.8, 0.6) circle (0.02) node[above right] {$(x, y)$};
		\draw[fill] (-0.8, 0.6) circle (0.02);% node[above right] {$(x, y)$};
		\draw[fill] (-0.8, -0.6) circle (0.02);% node[above right] {$(x, y)$};
		\draw[fill] (0.8, -0.6) circle (0.02);% node[above right] {$(x, y)$};

		\draw[->, dashed] (0.8, 0.6) -- (-0.8, 0.6) node[midway, below right] {$U_3$} node[left] {$P_3$};
		\draw[->, dashed] (0.8, 0.6) -- (-0.8, -0.6) node[midway, below right] {$U_2$} node[below left]{$P_2$};
		\draw[->, dashed] (0.8, 0.6) -- (0.8, -0.6) node[midway, below left] {$U_1$} node[below]{$P_1$};

\end{tikzpicture}
	\begin{tikzpicture}[scale=2, >=latex', auto]
		\draw[->](-1.2, 0) -- (1.2, 0) node[right]{$x$};
		\draw[->](0, -1.2) -- (0, 1.2) node[above]{$y$};
		\draw (0, 0) circle (1);
		\draw[ultra thick] (1, 0) arc (0:90:1);

		\draw[fill, dashed, ->] (-0.8, 0.6) node[left]{$P_3$} circle (0.025) -- (0.47, 0.88) circle (0.025) node[above right]{$M_3\cdot P$};
		\draw[fill, dashed, ->] (-0.8, -0.6) node[below left]{$P_2$} circle (0.025) -- (0.69, 0.72) circle (0.025) 
		node[midway, above left] {$H$} node[right]{$M_2\cdot P$};
		\draw[fill, dashed, ->] (0.8, -0.6) node[below]{$P_1$} circle (0.025) -- (0.92, 0.38) circle (0.025) node[right]{$M_1\cdot P$};

\end{tikzpicture}
\end{center}
\caption{Actions of $M_1, M_2, M_3$ on $(x, y)\in \QQQ$. First, $(x, y)$ moves to $P_1, P_2, P_3$ in the 4nd, 3rd, 2th quadrants under $U_1, U_2, U_3$, respectively.
These points are then brought back to $\QQQ$ under $H$.\label{ActionofMj}}
\end{figure}
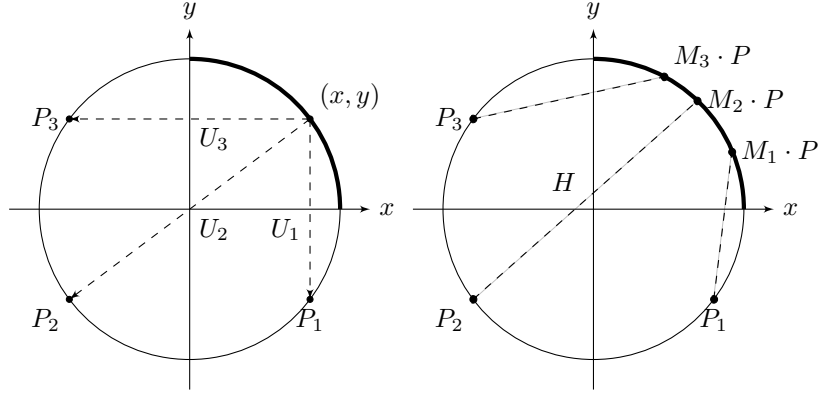

\subsection{Proof of Berggren's theorem}\label{SecBerggren}
Equipped with this geometric insight, we are now ready to prove Berggren's theorem. First, define a directed graph as follows:
\begin{description}
	\item[Vertex] the vertices are all primitive Pythagorean triples $\mathbf{v} = (a, b, c)$, and
	\item[Edge] there is an edge from $\mathbf{v} = (a, b, c)$ (which we regard as a \emph{parent}) to $\mathbf{v}' = (a', b', c')$ (a \emph{child}) whenever $\mathbf{v}'= M_j\mathbf{v}$ for $j=1, 2$, or $3$.
\end{description}
We show first that each vertex $\mathbf{v}$ has outdegree 3, in other words, all 
\[
	M_j\mathbf{v} = HU_j \mathbf{v}
\] 
for $j = 1, 2, 3$ are primitive Pythagorean triples. 
First, note that $U_j\mathbf{v}\in C_Q^+$ and $H$ leaves $C_Q^+$ stable. So, $M_j\mathbf{v}\in C_Q^+$.
We leave as an exercise for the reader to verify that, if $A\in O_Q(\mathbb{R})$ is \emph{integral}, namely, if $A\mathbf{w}\in \mathbb{Z}^3$ for any $\mathbf{w} = (w_1, w_2, w_3)\in \mathbb{Z}^3$ and if $\gcd(w_1, w_2, w_3) = 1$, then the gcd of the three coordinates of $A\mathbf{w}$ is one, as well.

Now, it remains to show $HU_j\mathbf{v}$ represents a point in $\QQQ$.
Clearly, the vector $U_j\mathbf{v}$ represents a point, say, $P_j$, which is not in $\QQQ$.
Then $H$ moves $P_j$ back to one of the three subarcs of $\QQQ$ (corresponding to the three Romik digits $d=1, 2, 3$), as described in as Figure~\ref{ActionofMj}.
This finishes proving the claim that $\mathbf{v}$ has outdegree 3.

\begin{figure}
\begin{center}
	\begin{tikzpicture}[scale=2, >=latex', auto]
		\draw[->](-1.2, 0) -- (1.2, 0) node[right]{$x$};
		\draw[->](0, -1.2) -- (0, 1.2) node[above]{$y$};
		\draw (0, 0) circle (1);
		\draw[ultra thick] (1, 0) arc (0:90:1);

		\draw[fill, dashed, ->] (0.47, 0.88) circle (0.02) node[above right] {$P$}
		-- (-0.8, 0.6) circle (0.02) node[midway, below left] {$H$}
		node[left] {$P_H$};
%		\draw[fill, dashed, ->] (-0.8, -0.6) circle (0.04) -- (0.69, 0.72) circle (0.04);
%		\draw[fill, dashed, ->] (0.8, -0.6) circle (0.04) -- (0.92, 0.38) circle (0.04);

%> 8/17
%[1] 0.4705882
%> 15/17
%[1] 0.8823529
%> 
%> 20/29
%[1] 0.6896552
%> 21/29
%[1] 0.7241379
%> 
%> 12/13
%[1] 0.9230769
%> 5/13
%[1] 0.3846154
		
\end{tikzpicture}
	\begin{tikzpicture}[scale=2, >=latex', auto]
		\draw[->](-1.2, 0) -- (1.2, 0) node[right]{$x$};
		\draw[->](0, -1.2) -- (0, 1.2) node[above]{$y$};
		\draw (0, 0) circle (1);
		\draw[ultra thick] (1, 0) arc (0:90:1);

		\draw[fill, dashed, ->] (-0.8, 0.6) circle (0.02) node[left] {$P_H$}
		-- (0.8, 0.6) circle (0.02) node[below right, midway] {$U_j$}
		node[right] {$M_j^{-1}\cdot P$};

\end{tikzpicture}
\end{center}
\caption{Action of $M_j^{-1}$ on $P\in \QQQ$. First, $H$ moves $P$ outside of $\QQQ$ to $P_H$. Then exactly one of $U_1, U_2, U_3$, namely, $U_j$ with $j=d(x, y)$ brings it back to $\QQQ$. \label{ActionofMjInverse}}
\end{figure}
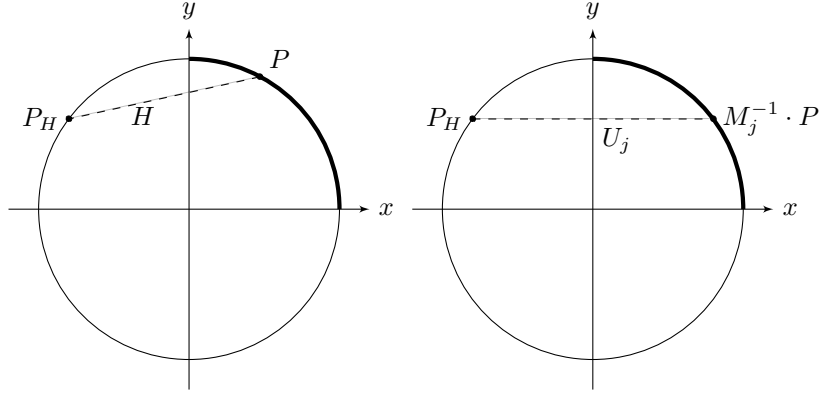

Next, suppose $\mathbf{v}$ is a vertex not equal to $(3, 4, 5)$ and $(4, 3, 5)$.
We argue that $\mathbf{v} = (a, b, c)$ has indegree 1, or, equivalently, 
\[
	M_j^{-1}\mathbf{v} = U_jH\mathbf{v}
\]
is a primitive Pythagorean triple for exactly one value of $j=1, 2, 3$.
In fact, if we let $j = d(x, y)$, the Romik digit of $P = (x, y)$ with $x=\frac ac, y=\frac bc$ (namely, $P$ is the point represented by $\mathbf{v}$),
then the point $P_H$ represented by $H\mathbf{v}$ is on the second, third or fourth quadrant, depending on $j=1, 2, 3$, respectively. 
Then, $U_j$ moves $P_H$ back to $\QQQ$, as in Figure~\ref{ActionofMjInverse}. 
For the exceptional case $\mathbf{v} = (3, 4, 5)$ and $(4, 3, 5)$, we have $U_jH\mathbf{v} = (1, 0, 1)$ and $(0, 1, 1)$, which are not (technically) primitive Pythagorean triples.
So, $(3, 4, 5)$ and $(4, 3, 5)$ are the only vertices of indegree 0.
\begin{figure}
\begin{center}
	
\tikzset{triarrow/.pic={
    \draw (0, 0) -- (0, 0.5) node[above]{$\vdots$};
    \draw (-0.1, 0) -- (-0.2, 0.5) node[above]{$\vdots$};
    \draw (0.1, 0) -- (0.2, 0.5) node[above]{$\vdots$};
  }
}
\begin{tikzpicture}[->, >=latex', auto]
\node (base1) at (.5, -1) {$(3, 4, 5)$};

\node (1f1) at (-1, 1) {$\bullet$};
\node (1f2) at (1, 1) {$\bullet$};
\node (1f3) at (3, 1) {$\bullet$};

\pic at (-1, 1.3) {triarrow};
\pic at (0, 1.3) {triarrow};
\pic at (1, 1.3) {triarrow};

\draw (base1) to (1f1);
\draw (base1) to (1f2);
\draw (base1) to (1f3);

\node (base2) at (2.5, -1) {$(4, 3, 5)$};

\node (2f1) at (0, 1) {$\bullet$};
\node (2f2) at (2, 1) {$\bullet$};
\node (2f3) at (4, 1) {$\bullet$};

\pic at (2, 1.3) {triarrow};
\pic at (3, 1.3) {triarrow};
\pic at (4, 1.3) {triarrow};

\draw (base2) to (2f1);
\draw (base2) to (2f2);
\draw (base2) to (2f3);

\draw (-2.5, 2.5) to[bend left] (5.5, 2.5) -- (1.5, -4) -- cycle;

\draw[dashed] (-2, -1) -- (-2, -3) node[midway, left]{$M_j^{-1}$};
\draw[dashed] (4, -3) -- (4, -1) node[midway, right]{$ M_1, \, \, M_2, \, \, M_3$};
\end{tikzpicture}

%	\begin{tikzpicture}[scale=1, >=latex', auto]
%	\draw[fill] (-1, 1) circle (0.05); 
%	\draw[fill] (0, 1) circle (0.05); 
%	\draw[fill] (1, 1) circle (0.05); 
%	\draw[fill] (0, 0) circle (0.08) node[right] {$(a, b, c)$}; 
%	\draw[fill] (0, -1) circle (0.05); 
%	\draw (-1, 1) -- (0, 0) -- (0, -1);
%	\draw (0, 1) -- (0, 0);
%	\draw (1, 1) -- (0, 0);
%
%	\draw[->, dashed] (-4, 0) -- (-4, -1) node[midway, left]{$T$};
%	\draw[->, dashed] (-2, 0) -- (-2, -1) node[midway, left]{$M_j^{-1}$};
%
%	\draw[->, dashed] (2, 0) -- (2, 1) node[midway, right]{$ M_1, \, \, M_2, \, \, M_3$};
%	\end{tikzpicture}
\end{center}
      \caption{The vertices are drawn on $C_Q^+$. Except for $(3, 4, 5)$ and $(4, 3, 5)$, each vertex is of outdegree 3 and of indegree 0. The actions of $M_1, M_2, M_3$ move the vertices upwards, increasing the $x_3$-coordinates, while $M_j^{-1}$ move them downwards, decreasing the $x_3$-coordinates.\label{FunnelVertex}}
\end{figure}
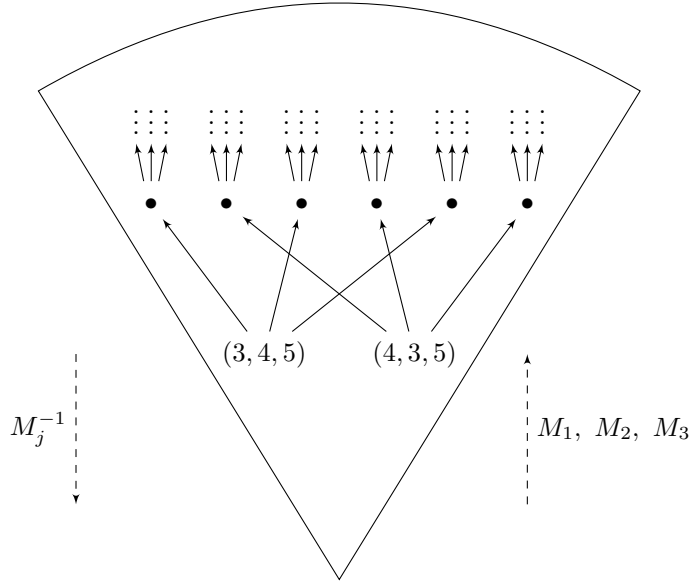

Lastly, we claim that the actions of $M_1, M_2, M_3$ on a vertex $\mathbf{v} =(a, b, c)$ increase its $x_3$-coordinate, while the actions of $M_j^{-1}$ decrease it.
For any $j$, let $(a', b', c') = M_j\mathbf{v} = U_jH\mathbf{v}$.
Since $\mathbf{v}$ represents a point in $\QQQ$ we see that $H\mathbf{v}$ has greater $x_3$-coordinate than $\mathbf{v}$, as was observed before. 
Also, $U_j$ always preserves the $x_3$-coordinates. This proves $c'>c$. 
Additionally, we note $\mathbf{v} = M_j^{-1}(a', b', c')$ (when $j$ is the Romik digit of the point represented by $(a', b', c')$.) This proves the claim. 

To summarize, we proved that each primitive Pythagorean triple $(a, b, c)$ (except for $(3, 4, 5)$ and $(4, 3, 5)$) satisfies the three properties 
\begin{itemize}
	\item $(a, b, c)$ has outdegree 3,
	\item $(a, b, c)$ has indegree 1, 
	\item any vertex coming out of $(a, b, c)$ has its $x_3$-coordinate greater than $c$, while the vertex going into $(a, b, c)$ has a smaller $x_3$-coordinate than $c$.
\end{itemize}

We will say that a vertex $(a, b, c)$ is a \emph{funnel} vertex, if it satisfies the above three properties.
The idea behind this terminology is this. 
Suppose $P\in \QQQ$. 
If we trace the vertices $\mathbf{v}$ on $C_Q^+$ representing the points $P$, $T(P)$, $T^2(P), \dots,$ then we must move from a vertex to another along the \emph{opposite} direction in the tree.
So, at any vertex (except for $(3, 4, 5)$ and $(4, 3, 5)$), a $T$-orbit always moves \emph{towards} the origin.

What we proved so far is summarized as
\begin{theorem}\label{BerggrenTheorem}
	All primitive Pythagorean triples $(a, b, c)$ are funnel vertices, except for $(3, 4, 5)$ and $(4, 3, 5)$.
\end{theorem}
Now, Berggren's theorem is an easy corollary to this, because any given primitive Pythagorean triple $(a, b, c)$, under the successive actions of $M_j^{-1}$, will funnel down until it reaches $(3, 4, 5)$ or $(4, 3, 5)$ and it must do so in a unique way.

\section{Proof of Lagrange's Theorem}\label{SecProof}
The goal of this section is to prove
\begin{theorem}[Lagrange's Theorem for the Romik system]\label{LagrangeTheorem}
Let $P = (\alpha, \beta)\in \QQQ$. Its Romik digit expansion
\[
(\alpha, \beta) = [d_1, d_2, \dots, ]_\QQQ
\]
is eventually periodic if and only if $P$ is defined over a real quadratic field, namely, there exists a squarefree positive integer $D$ such that $\alpha, \beta\in \mathbb{Q}(\sqrt D)$.
\end{theorem}

\subsection{$Q$-cross product}\label{SectionTwistedProduct}
Let $\mathbf{v}_1, \mathbf{v}_2$ be two vectors in $\mathbb{R}^3$. 
We define the \emph{$Q$-cross product} $\mathbf{v}_3 = \mathbf{v}_1 \times_Q \mathbf{v}_2$ by
\begin{equation}\label{TwistedCrossProduct}
\mathbf{v}_3 =  M_Q^{-1} (\mathbf{v}_1 \times \mathbf{v}_2),
\end{equation} 
where $\mathbf{v}_1 \times \mathbf{v}_2$ is the usual cross product in $\mathbb{R}^3$ and $M_Q$ is the symmetric matrix associated with the quadratic form $Q(\mathbf{x})$. 
For the case of Pythagorean form $Q(\mathbf{x}) = x_1^2 + x_2^2 - x_3^3$, which is all we need in this paper, 
\begin{equation}\label{CrossProductCoordinates}
\mathbf{v}_1 \times_Q \mathbf{v}_2
= 
(
 b_1c_2 - b_2c_1,  
 a_2c_1 - a_1c_2,  
 a_2b_1 - a_1b_2
),
\end{equation}
if we write $\mathbf{v}_1 = (a_1, b_1, c_1)$ and $\mathbf{v}_2 = (a_2, b_2, c_2)$.

We deduce from \eqref{TripleCorrespondence} and \eqref{TwistedCrossProduct} that 
$\mathbf{v}_3 = \mathbf{v}_1 \times_Q \mathbf{v}_2$ satisfies
\begin{enumerate}
\item[(A)] $\langle \mathbf{v}_1, \mathbf{v}_3 \rangle  = \langle \mathbf{v}_2, \mathbf{v}_3 \rangle = 0$, and
\item[(B)] $Q(\mathbf{v}_3) = [\mathbf{v}_1, \mathbf{v}_2, \mathbf{v}_3]$, where $[\mathbf{v}_1, \mathbf{v}_2, \mathbf{v}_3]:=(\mathbf{v}_1 \times \mathbf{v}_2) \cdot \mathbf{v}_3$ is by definition the triple product of $\mathbf{v}_1, \mathbf{v}_2, \mathbf{v}_3$.
\end{enumerate} 
For example,
\[
\langle 
\mathbf{v}_1, \mathbf{v}_3
\rangle
=
\mathbf{v}_1^T M_Q M_Q^{-1} (\mathbf{v}_1 \times \mathbf{v}_2) = 0
\]
and likewise $\langle\mathbf{v}_2, \mathbf{v}_3\rangle=0$. Also,
\[
Q(\mathbf{v}_3) = 
\langle 
\mathbf{v}_3, \mathbf{v}_3
\rangle =
\mathbf{v}_3^T M_Q \,M_Q^{-1} (\mathbf{v}_1 \times \mathbf{v}_2) =
\mathbf{v}_3\cdot (\mathbf{v}_1 \times \mathbf{v}_2) =
(\mathbf{v}_1 \times \mathbf{v}_2) \cdot \mathbf{v}_3.
\]
Conversely, if $\mathbf{v}_1$ and $\mathbf{v}_2$ are linearly independent, then
the properties (A) and (B) characterizes $\mathbf{v}_1\times_Q\mathbf{v}_2$ uniquely.
First, the vectors satisfying (A) form a one-dimensional subspace of $\mathbb{R}^3$, from which we choose a (nonzero) vector $\mathbf{u}$. Then, we take $\mathbf{v}_3 = k\mathbf{u}$ with
$k = [\mathbf{v}_1, \mathbf{v}_2, \mathbf{u}]/Q(\mathbf{u})$. It is easily verified that $\mathbf{v}_3$ satisfies (B).
%(If $\mathbf{v}_1$ and $\mathbf{v}_2$ are linearly dependent then $\mathbf{v}_1\times_Q\mathbf{v}_2= 0$.)

Even though we do not need it in this paper, we present a more functorial construction of $\mathbf{v}_1 \times_Q \mathbf{v}_2$ for future reference.
What we discuss in this paragraph is standard. See, for example, \cite{Dar94}.
Let $V$ be a $n$-dimensional vector space (over $\mathbb{R}$), equipped with a nondegenerate bilinear pairing $\langle \cdot, \cdot \rangle$.
Then for each $p = 1, \dots, n$, the bilinear pairing $\langle \cdot, \cdot \rangle_p$ of $\bigwedge^p V$, which is defined by
\[
\langle
\mathbf{v}_1 \wedge \cdots \wedge \mathbf{v}_p, 
\mathbf{w}_1 \wedge \cdots \wedge \mathbf{w}_p
\rangle_p
= \det (
\langle
\mathbf{v}_i,
\mathbf{w}_j
\rangle
)_{i, j = 1, \dots, p},
\]
is also nondegenerate. 
Fix a basis of $\bigwedge^n V$ (that is, an orientation or a volume form), say, $\tau \in \bigwedge^n V$. 
Then, for any $\lambda \in \bigwedge^p V$, there exists a unique element of $\bigwedge^{n-p} V$, which we denote by $\star \lambda$, satisfying
\[
\lambda \wedge \mu =
\langle
\star \lambda,
\mu
\rangle_p
\,
\tau
\]
for all $\mu \in \bigwedge^{n-p} V$. This way, we define \emph{the Hodge star operator}
\[
\bigwedge^p V 
\longrightarrow 
\bigwedge^{n-p} V, \qquad \lambda \mapsto \star\lambda.
\]
The relevance of this in our context is this. 
Let $V = \mathbb{R}^3$, equipped with the pairing $\langle \cdot, \cdot \rangle$ given by the Pythagorean form $Q(\mathbf{x})$ and choose the standard orientation $\tau = \mathbf{e}_1\wedge \mathbf{e}_2 \wedge \mathbf{e}_3$.
Then one can show as an exercise that 
\begin{equation}\label{HodgeStar}
\mathbf{v}_1 \times_Q \mathbf{v}_2 = \star(\mathbf{v}_1 \wedge \mathbf{v}_2),
\end{equation}
because $\star(\mathbf{v}_1 \wedge \mathbf{v}_2)$ satisfies the properties (A) and (B) above.

\begin{proposition}\label{OrthogonalActionCrossProduct}
For any $A\in O_Q(\mathbb{R}^3)$ and $\mathbf{v}_1, \mathbf{v}_2 \in \mathbb{R}^3$, 
\[
A\mathbf{v}_1\times_Q A\mathbf{v}_2 = (\det A) A(\mathbf{v}_1\times_Q\mathbf{v}_2).
\]
\end{proposition}
\begin{proof}
First of all, if $\mathbf{v}_1$ and $\mathbf{v}_2$ are linearly dependent, then both sides of the above equation are zero. So, we assume that $\mathbf{v}_1$ and $\mathbf{v}_2$ are linearly independent and we prove that
\[
\mathbf{v} := (\det A) A(\mathbf{v}_1\times_Q\mathbf{v}_2)
\]
satisfies the properties (A) and (B) above with respect to $A\mathbf{v}_1$ and $A\mathbf{v}_2$. For (A),
\[
\langle 
A\mathbf{v}_1, \mathbf{v}
\rangle =
(\det A)
\langle 
A\mathbf{v}_1, A(\mathbf{v}_1\times_Q\mathbf{v}_2)
\rangle =
(\det A)
\langle 
\mathbf{v}_1, \mathbf{v}_1\times_Q\mathbf{v}_2
\rangle =0.
\]
Likewise, $\langle 
A\mathbf{v}_2, \mathbf{v}
\rangle =0.$
For (B), we note in general that
\[
[A\mathbf{w}_1, A\mathbf{w}_2, A\mathbf{w}_3] = (\det A)[\mathbf{w}_1, \mathbf{w}_2, \mathbf{w}_3],
\]
for any $A\in O_Q(\mathbb{R})$ and $\mathbf{w}_1, \mathbf{w}_2, \mathbf{w}_3\in \mathbb{R}^3$.
So, 
\begin{align*}
[A\mathbf{v}_1, A\mathbf{v}_2, \mathbf{v}] &=
[A\mathbf{v}_1, A\mathbf{v}_2, (\det A) A(\mathbf{v}_1\times_Q\mathbf{v}_2)] =
[\mathbf{v}_1, \mathbf{v}_2, \mathbf{v}_1\times_Q \mathbf{v}_2] \\
&= Q(\mathbf{v}_1\times_Q \mathbf{v}_2) = Q(\mathbf{v}),
\end{align*}
where the last equality is because $A$ is orthogonal with respect to $Q(\mathbf{x})$. 
This proves (B).
\end{proof}

\begin{proposition}\label{Qthree}
Let $\mathbf{v}_1, \mathbf{v}_2 \in \mathbb{R}^3$ and $\mathbf{v}_3 = \mathbf{v}_1\times_Q \mathbf{v}_2$. Assume $\mathbf{v}_1\in C_Q$, that is, $Q(\mathbf{v}_1) = 0$. Then
\[
Q(\mathbf{v}_3)  = (
\langle 
\mathbf{v}_1, \mathbf{v}_2
\rangle
)^2.
\]
\end{proposition}
\begin{proof}
It is possible to give a ``coordinate-free'' proof based on \eqref{HodgeStar}.
Instead, we present an elementary approach beginning with \eqref{CrossProductCoordinates}. 
As the validity of the equation in this proposition remains unchanged under rescaling of $\mathbf{v}_1$ and $\mathbf{v}_2$, without loss of generality, we may assume
\[
\mathbf{v}_1 = (\alpha, \beta, 1), \text{ and }
\mathbf{v}_2 = (\alpha', \beta', 1),
\]
with $\alpha^2 + \beta^2 = 1$.
(Later in the proof, we separately deal with the case when the $x_3$-coordinate of $\mathbf{v}_2$ is zero.)
Then \eqref{CrossProductCoordinates} gives
$\mathbf{v}_3 = (\beta - \beta', \alpha' - \alpha, \alpha'\beta - \alpha\beta')$,
so
	\begin{align*}
		Q(\mathbf{v}_3) &= (\beta - \beta')^2 + (\alpha'-\alpha)^2 - (\alpha'\beta - \alpha\beta')^2 \\
			   &= 1 + (\alpha'^2 + \beta'^2) - 2(\alpha\alpha' + \beta\beta') - (\alpha'^2\beta^2 + \alpha^2\beta'^2 - 2\alpha\alpha'\beta\beta') \\ 
			   &= 1 + (\alpha'^2 + \beta'^2) - 2(\alpha\alpha' + \beta\beta') - (\alpha'^2(1 - \alpha^2) +(1 - \beta^2)\beta'^2 - 2\alpha\alpha'\beta\beta') \\ 
			   &= 1 - 2(\alpha\alpha' + \beta\beta') + (\alpha\alpha'+\beta\beta')^2 \\ 
	     &= (\langle \mathbf{v}_1, \mathbf{v}_2\rangle )^2.
	\end{align*}
For the case $\mathbf{v}_2 = (\alpha', \beta', 0)$, we have $\mathbf{v}_3 = ( - \beta', \alpha', \alpha'\beta - \alpha\beta')$. And, the rest is similar to the above calculation. We omit the detaill.
\end{proof}

Assume further that $\mathbf{v}_1, \mathbf{v}_2 \in C_Q$ and that $\mathbf{v}_1$ and $\mathbf{v}_2$ are linearly independent.
Then $\mathbf{w}:= \mathbf{v}_1\times_Q \mathbf{v}_2$ has a simple geometric description.
The two subspaces
\[
\langle
\mathbf{v}_1
\rangle^{\perp} := 
\{ \mathbf{x}\in\mathbb{R}^3 \mid \langle \mathbf{v}_1, \mathbf{x} \rangle = 0 \}
\]
and
\[
	\langle \mathbf{v}_2 \rangle^{\perp} := 
\{ \mathbf{x}\in\mathbb{R}^3 \mid \langle \mathbf{v}_2, \mathbf{x} \rangle = 0 \}
\]
are tangent to $C_Q$ and their intersection is the line spanned by $\mathbf{w}$. 
Also, we have $\langle \mathbf{v}_1, \mathbf{v}_2 \rangle \neq 0$, otherwise the span of $\mathbf{v}_1$ and $\mathbf{v}_2$ would be contained in $C_Q$, contradicting the nondegeneracy of $Q(\mathbf{x})$. 
So, Proposition~\ref{Qthree} shows that $W:=Q(\mathbf{w})>0$. 
In particular, $\mathbf{w}$ lies on the \emph{one-sheeted hyperboloid}
\begin{equation}\label{Hyperboloid}
		\mathcal{H}(W): x_1^2 + x_2^2 - x_3^2 = W.
	\end{equation}

Keeping the assumption that $\mathbf{v}_1, \mathbf{v}_2 \in C_Q$, we additionally suppose that $\mathbf{v}_1$ represents a point in $\QQQ$ via the projection map $\pi$, which is defined in \eqref{ProjectionMap}.
In this case, the following proposition says that $\mathbf{v}_1 \times_Q \mathbf{v}_2$ represents a point in $\mathbb{R}^2$ ``close to'' $\QQQ$. 
\begin{proposition}\label{TwistedCrossProductIsClose}
	Suppose $\mathbf{v}_1, \mathbf{v}_2 \in C_Q$ are linearly independent. Assume that $\mathbf{v}_1$ represents a point in $\QQQ$.
	Write $\mathbf{w}:=(w_1, w_2, w_3) = \mathbf{v}_1 \times_Q \mathbf{v}_2$ and let $W = Q(\mathbf{w})$.
	\begin{enumerate}
		\item If $w_3 > 0$, then $w_1 > - \sqrt W$ and $w_2 > -\sqrt W$.
		\item If $w_3 < 0$, then $w_1 < \sqrt W$ and $w_2 < \sqrt W$.
	\end{enumerate}
\end{proposition}
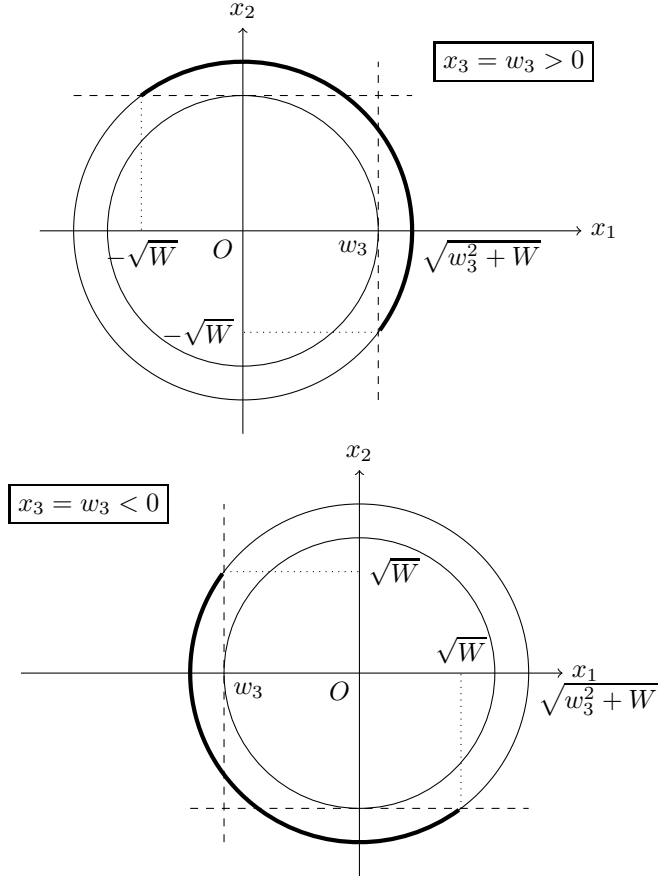
\begin{figure}
\begin{center}
	\begin{tikzpicture}[scale=0.9]
\draw[->] (-3,0)--(5,0) node[right]{$x_1$};
\draw[->] (0,-3)--(0,3) node[above]{$x_2$};
\node at (4, 2.5) {$\boxed{x_3 = w_3>0}$};
\node[below left] at (0, 0) {$O$};

\draw[] (0, 0) circle (2);
\draw[] (0, 0) circle (2.5);

\draw[dashed] (-2.5, 2) -- (2.5, 2);
%\draw[dashed] (-2.2, 2.2) -- (2.2, 2.2);
\draw[dashed] (2, -2.5) -- (2, 2.5);

\draw[ultra thick] ([shift=(-36:2.5)]0, 0) arc (-36:127:2.5);
%\draw[ultra thick] ([shift=(23.5:2.2)]0, 0) arc (23.5:67.5:2.2);

\draw[dotted] (-1.5, 2) -- (-1.5, 0) node [below] {$-\sqrt W$};
\draw[dotted] (2, -1.5) -- (0, -1.5) node [left] {$-\sqrt W$};
%\draw[->] (-2.2, 3.2) -- (-2.2, 2.2);
%\node[right] at (-2.2, 2.2) {$<\sqrt D$};

%\draw[<->] (-1.5, 3) -- (0, 3);
%\node[above] at (-0.75, 3) {$<\sqrt{\epsilon}$};

\node[below left] at (2, 0) {$w_3$};
\node[below right] at (2.5, 0) {$\sqrt{w_3^2 + W}$};
\end{tikzpicture}
\hfill
\begin{tikzpicture}[scale=0.9]
\draw[->] (-5,0)--(3,0) node[right]{$x_1$};
\draw[->] (0,-3)--(0,3) node[above]{$x_2$};
\node at (-4, 2.5) {$\boxed{x_3 = w_3<0}$};
\node[below left] at (0, 0) {$O$};

\draw[] (0, 0) circle (2);
\draw[] (0, 0) circle (2.5);

\draw[dashed] (-2.5, -2) -- (2.5, -2);
%\draw[dashed] (-2.2, 2.2) -- (2.2, 2.2);
\draw[dashed] (-2, -2.5) -- (-2, 2.5);

\draw[ultra thick] ([shift=(144:2.5)]0, 0) arc (144:306:2.5);
%\draw[ultra thick] ([shift=(23.5:2.2)]0, 0) arc (23.5:67.5:2.2);

\draw[dotted] (1.5, -2) -- (1.5, 0) node [above] {$\sqrt W$};
\draw[dotted] (-2, 1.5) -- (0, 1.5) node [right] {$\sqrt W$};
%\draw[->] (-2.2, 3.2) -- (-2.2, 2.2);
%\node[right] at (-2.2, 2.2) {$<\sqrt D$};

%\draw[<->] (-1.5, 3) -- (0, 3);
%\node[above] at (-0.75, 3) {$<\sqrt{\epsilon}$};

\node[below right] at (-2, 0) {$w_3$};
\node[below right] at (2.5, 0) {$\sqrt{w_3^2 + W}$};
\end{tikzpicture}
\end{center}
\caption{The sections of $C_Q$ (the inner circle) and of $\mathcal{H}(W)$ (the outer one) by $x_3= w_3$. The first picture is for $w_3 > 0$ and the second is for $w_3 < 0$. \label{MainPictureI}}
\end{figure}
\begin{proof}
We will consider the sections of $C_Q$ and of $\mathcal{H}(W)$ by the plane $x_3 = w_3$, which appear as the concentric circles depicted in Figure~\ref{MainPictureI}. First, asssume $w_3 > 0$. 
Recall that $\mathbf{w}$ lies in the intersection of $\langle \mathbf{v}_1 \rangle^{\perp}$ and $\langle \mathbf{v}_2 \rangle^{\perp}$.
Note that $\langle \mathbf{v}_1 \rangle^{\perp}$ appears as a tangent line to the inner circle (the section of $C_Q$ by $x_3 = w_3$) at $\mathbf{v}_1$. 
Since $x, y \ge 0$ the intersection of $\langle \mathbf{v}_1 \rangle^{\perp}$ and $\langle \mathbf{v}_2 \rangle^{\perp}$ can lie only on the thick arc of the outer circle in Figure~\ref{MainPictureI}. 
It is easy to verify that the thick arc corresponds to the stated inequalities for $w_1$ and $w_2$. The case for $w_3<0$ is similar.
\end{proof}

The next proposition describes how the integrality of $\mathbf{v}$ is reflected in the twisted product. 

\begin{proposition}\label{LatticeTheorem}
Suppose that $\mathbf{v}\in C_Q$ is defined over a real quadratic field $K = \mathbb{Q}(\sqrt D)$, whose ring of integers will be denoted by $\mathcal{O}_K$. Assume that $\mathbf{v}$ and $\mathbf{v}^{\sigma}$ are linearly independent. If $\mathbf{v}\in (\mathcal{O}_K)^3$, then
\[
\mathbf{v} \times_Q \mathbf{v}^{\sigma} \in (\sqrt D \mathbb{Z}/2)^3.
\]
\end{proposition}
\begin{proof}
	Let $ \mathbf{w} = \mathbf{v} \times_Q \mathbf{v}^{\sigma}$.
It is easy to see from \eqref{CrossProductCoordinates} that, whenever $\mathbf{v}$ is in $(\mathcal{O}_K)^3$, $\mathbf{w}$ is also in $(\mathcal{O}_K)^3$. 
At the same time, observe that
\[
	\mathbf{w}^{\sigma} = \mathbf{v}^{\sigma} \times_Q \mathbf{v} = -\mathbf{w}.
\]
Hence, we have $\mathbf{w}\in (\sqrt D \mathbb{Q})^3$, where $\sqrt D \mathbb{Q}:= \{ \sqrt D r\mid r \in \mathbb{Q} \}$. 
This finishes the proof because $\sqrt D \mathbb{Q} \cap \mathcal{O}_K \subseteq \sqrt{D}\mathbb{Z}/2$.
\end{proof}

\subsection{Periodicity is sufficient for quadratic irrationality}\label{SecNecessaryCondition}
We first prove in this subsection that, if $P = (\alpha, \beta)\in \QQQ$ has an eventually periodic Romik digit expansion, then $P$ is defined over a real quadratic extension of $\mathbb{Q}$. 
Suppose that
\[
(\alpha, \beta) = [f_1, f_2, \dots, f_l, \overline{d_1, \dots, d_k}]_\QQQ
=
[f_1, f_2, \dots, f_l, d_1, \dots, d_k, d_1, \dots, d_k,  \dots]_\QQQ.
\]
Define $P'= (\alpha', \beta') := T^l(P)$. Then clearly,
\[
(\alpha', \beta') = [\overline{d_1, \dots, d_k}]_\QQQ
\]
is \emph{purely} periodic. Moreover, if we define $\bm{\xi}:= (\alpha, \beta, 1)\in C_Q$, then $\bm{\xi}$ represents $P$, consequently, the vector
\[
M_{f_l}^{-1}\cdots M_{f_1}^{-1}\bm{\xi}
\]
represents $P'$.
This argument shows that, if $\alpha', \beta'$ are in a real quadratic field $K$, then $\alpha, \beta\in K$ as well. 
Thus, it is enough to prove the statement for the purely periodic case.
Therefore, we will assume from now on that 
\[
P = (\alpha, \beta) = [\overline{d_1, \dots, d_k}]_\QQQ
\]
is purely periodic. Then the quadratic irrationality of $P$ is an immediate consequence of Theorem~\ref{EigenvectorTheorem} below. 
We give an easy lemma first.

\begin{lemma}\label{Eigenvalues}
Suppose that $A\in O_Q(\mathbb{R}^3)$. 
Let $\mathbf{v}_1$ and $\mathbf{v}_2$ be eigenvectors of $A$, satisfying $\langle \mathbf{v}_1, \mathbf{v}_2 \rangle \neq 0$, and let $\lambda_1, \lambda_2$ be their eigenvalues respectively. Then $\lambda_1 \lambda_2 = 1$. In particular, if $\mathbf{v}$ is an eigenvector of $A$ with $Q(\mathbf{v}) \neq 0$, then its eigenvalue is $1$ or $-1$. 
\end{lemma}
\begin{proof}
	The first statement is an easy consequence of the identity
\[
\langle \mathbf{v}_1, \mathbf{v}_2 \rangle =
\langle A\mathbf{v}_1, A\mathbf{v}_2 \rangle =
\langle \lambda_1 \mathbf{v}_1, \lambda_2 \mathbf{v}_2 \rangle =
\lambda_1 \lambda_2 \langle \mathbf{v}_1, \mathbf{v}_2 \rangle,
\]
and the second statement follows from this by letting $\mathbf{v}_1 = \mathbf{v}_2$. 
\end{proof}
\begin{theorem}\label{EigenvectorTheorem}
Let
\[ 
	M = M_{d_1} \cdots M_{d_k}
\]
for some $d_1, \dots, d_k \in \{ 1, 2, 3 \}$. Assume that not all $d_1 = \cdots = d_k = 1$ and not all $d_1 = \cdots = d_k = 3$. Or, equivalently, $M\neq M_1^k$ and $M \neq M_3^k$. Then there exists a real quadratic field extension $K$ of $\mathbb{Q}$, satisfying the following properties: (we denote by $\sigma$ the nontrivial Galois automorphism of $K$ over $\mathbb{Q}$.) 
\begin{enumerate}
	\item One of the eigenvalues for $M$ is, say, $\lambda_1 > 1$ with $\lambda_1 \in K$. Another eigenvalue for $M$ is $\lambda_2 := \lambda_1^{\sigma}$, with $0< \lambda_2 < 1$. The last eigenvalue $\lambda_3 = \det(M)$, which is either $1$ or $-1$.
	\item A (nonzero) eigenvector associated to $\lambda_1$ is $Q$-null, which represents a  point $(\alpha, \beta)$ in $\QQQ$ with $\alpha, \beta\in K$.  We write $\mathbf{v}_1 := (\alpha, \beta, 1)$.
	\item $\mathbf{v}_2 := (\alpha^{\sigma}, \beta^{\sigma}, 1)$ is an eigenvector associated to $\lambda_2$.
	\item $\mathbf{v}_3:= \mathbf{v}_1 \times_Q\mathbf{v}_2$ is an eigenvector associated to $\lambda_3$.
\end{enumerate}
\end{theorem}
\begin{proof}
For any finite digit sequence $d_1, \dots d_k$, let us define \emph{the cylinder set}
\[
C(d_1, \dots, d_k) = \{ (x, y) \in \QQQ \mid d(T^{j-1}(x, y)) = d_j \text{ for }j = 1, \dots, k \}.
\]
Then $C(d_1, \dots, d_k)$ is just the closed sub-arc of $\QQQ$ whose endpoints are represented by
\[
M\begin{pmatrix}0 \\ 1\\1 \\ \end{pmatrix}
\qquad
\text{ and }
\qquad
M\begin{pmatrix}1 \\ 0\\1 \\ \end{pmatrix}.
\]
In particular, $C(d_1, \dots, d_k)$ is contained in the \emph{interior} of $\QQQ$ unless all of $d_j$ are $1$, or all of $d_j$ are $3$.  

Consider the continuous map from the unit circle $\UUU$ to itself, defined by
\[
(x, y) \mapsto 
\begin{pmatrix}x \\ y\\1 \\ \end{pmatrix}\mapsto 
M\begin{pmatrix}x \\ y\\1 \\ \end{pmatrix}
=
\begin{pmatrix}v_1 \\ v_2\\v_3 \\ \end{pmatrix}
\mapsto \left(\frac{v_1}{v_3}, \frac{v_2}{v_3}\right).
\]
This map sends $\QQQ$ to $C(d_1, \dots, d_k)$, therefore, must have at least one fixed point, say, $(\alpha, \beta)$ in $C(d_1, \dots, d_k)$.
Then $\mathbf{v}_1:= (\alpha, \beta, 1)$ is an eigenvector of $M$. 
Call its eigenvalue $\lambda_1$, then $\lambda_1 > 1$ because $M$ must increase the $x_3$-coordinates of the vectors representing any point in $\QQQ$. (See Figure~\ref{FunnelVertex}.)
Applying the same argument to the complement of $C(d_1, \dots, d_k)$ and $M^{-1}$, we see that there is another eigenvector $\mathbf{v}_2 :=(\alpha', \beta', 1)$ with $(\alpha', \beta')\not\in\QQQ$ associated to an eigenvalue, say, $\lambda_2$.

So far, we found two linearly independent eigenvectors $\mathbf{v}_1, \mathbf{v}_2$ of $M$, both of which are $Q$-null. 
Note that $\langle \mathbf{v}_1, \mathbf{v}_2 \rangle \neq 0$, otherwise the $\mathbb{R}$-linear span of $\mathbf{v}_1$ and $\mathbf{v}_2$ would be a subset of the $Q$-null set $C_Q$, contradicting the nondegeneracy of $Q(\mathbf{x})$.
Then Lemma~\ref{Eigenvalues} shows that $\lambda_1\lambda_2 = 1$.  
Moreover, Proposition~\ref{OrthogonalActionCrossProduct} implies that $\mathbf{v}_3:= \mathbf{v}_1 \times_Q \mathbf{v}_2$ is an eigenvector associated with the third eigenvalue $\lambda_3 = \det(M) = \pm 1$. 
Observe that the characteristic polynomial of $M$ is a monic polynomial with integer coefficients and its constant term is $\det(M)=\pm1$. So, $\lambda_1$ and $\lambda_2$ must be quadratic irrationals and Galois conjugates one another. The remaining statements in the theorem follow from this easily.
\end{proof}

Even though it is not needed for our proof of Lagrange's theorem, we obtain as a corollary to Theorem~\ref{EigenvectorTheorem} an analogue of Galois' theorem for the Romik system.
Recall that a classical theorem by Galois \cite{Per29} relates the continued fraction expansion of a quadratic irrational to that of its Galois conjugate. 

\begin{corollary}[Galois' Theorem for the Romik system]\label{Galois}
Suppose that
\[
P = (\alpha, \beta) = [\overline{d_1, \dots, d_k}]_{\QQQ},
\]
and write $P^{\sigma} = (\alpha^{\sigma}, \beta^{\sigma})$ for its Galois conjugate. Then
\[
\begin{cases}
\alpha^{\sigma} > 0\text{ and } \beta^{\sigma} < 0 & \text{ if } d_k=1, \\
\alpha^{\sigma} < 0\text{ and } \beta^{\sigma} < 0 & \text{ if } d_k=2, \\
\alpha^{\sigma} < 0 \text{ and } \beta^{\sigma} > 0 & \text{ if } d_k=3. \\
\end{cases}
\]
Moreover, 
\[
(|\alpha^{\sigma}|, |\beta^{\sigma}|) = [\overline{d_{k-1}, d_{k-2}, \dots, d_2, d_1, d_k}]_{\QQQ}.
\]
\end{corollary}
\begin{proof}
Let $\mathbf{v}_1 = (\alpha, \beta, 1)$ as in Theorem~\ref{EigenvectorTheorem} and write
\[
	M_{d_1} \cdots M_{d_k} \mathbf{v}_1 = \lambda_1 \mathbf{v}_1.
\]
Conjugating this equation with $\sigma$ and then rewriting this using \eqref{MandHU} and \eqref{MInverseandHU},
\begin{equation}\label{UkTrick}
	HU_{d_1} HU_{d_2} \cdots \cdots HU_{d_k} \mathbf{v}_1^{\sigma} = 
	\lambda_1^{\sigma} \mathbf{v}_1^{\sigma}.
\end{equation}
Now, let $\mathbf{v}' := U_{d_k} \mathbf{v}_1^{\sigma}$ and multiply \eqref{UkTrick} by $U_{d_k}$ (from the left) to obtain
\begin{equation}\label{Mprime}
	M_{d_{k-1}}M_{d_{k-2}} \cdots M_{d_2}M_{d_1}M_{d_k} \mathbf{v}' = \lambda_1 \mathbf{v}'.
\end{equation}
Here, we use the fact $\lambda_1^{\sigma} = 1/\lambda_1$ from Theorem~\ref{EigenvectorTheorem}.
The equation \eqref{Mprime} then shows that $\mathbf{v}'$ is an eigenvector of 
\[
M':=M_{d_{k-1}}M_{d_{k-2}} \cdots M_{d_2}M_{d_1}M_{d_k},
\]
associated with the eigenvalue $\lambda_1 > 1$. From Theorem~\ref{EigenvectorTheorem} again, we conclude that the $x_1$- and $x_2$-coordinates of $\mathbf{v}'$ are positive. 
Thus, $U_{d_k} \mathbf{v}_1^{\sigma}$ has positive $x_1$- and $x_2$-coordinates, so that $U_{d_k} \mathbf{v}_1^{\sigma} = (|x^{\sigma}|, |y^{\sigma}|, 1)$.
This completes the proof.
\end{proof}

\subsection{Periodicity is necessary for quadratic irrationality}
Let $P= (\alpha,\beta)\in\QQQ$ be defined over a (real) quadratic field $K= \mathbb{Q}(\sqrt D)$. 
We prove that the Romik digit expansion of $P$ is eventually periodic, completing the proof of Lagrange's theorem for the Romik system (Theorem~\ref{LagrangeTheorem}).
Note that, if $P$ is rational, its Romik digit expansion ends with infinite succession of 1's or 3's, thus is periodic. 
So, we assume in what follows that $P$ is quadratic irrational.

We fix $\mathbf{v}_0 = (v_1, v_2, v_3)\in (\mathcal{O}_K)^3$ representing $P$.
And define $\mathbf{v}_n$ for each $n\ge1$ to be
\[
\mathbf{v}_n
=
M_{d_n}^{-1} \cdots M_{d_1}^{-1}
\mathbf{v}_0,
\]
so that $\mathbf{v}_n$ represents $T^n(P)$ for all $n\ge 1$. 
Finally, we define
\[
\mathbf{w}_n = 
\mathbf{v}_n \times_Q \mathbf{v}_n^{\sigma},
\]
for all $n\ge0$.
Recall that $\det(M_1) = \det(M_3) = 1$ and $\det(M_2) = -1$. So,
if we denote by $\epsilon_n$ the count of how many times the digit 2 appears in the sequence $\{d_1, \dots, d_n\}$, 
then $\det(M_{d_n}^{-1} \cdots M_{d_1}^{-1}) = (-1)^{\epsilon_n}$,
and
Proposition~\ref{OrthogonalActionCrossProduct} shows 
\begin{equation}\label{WnSequence}
	\mathbf{w}_n =  (-1)^{\epsilon_n}M_{d_n}^{-1} \cdots M_{d_1}^{-1}\mathbf{w}_0,
\end{equation}
for each $n\ge1$.
Note that $\mathbf{v}_0$ and $\mathbf{v}_0^{\sigma}$ are linearly independent because they represents two distinct points $(\alpha, \beta)$ and $(\alpha^{\sigma}, \beta^{\sigma})$. Therefore, $\mathbf{w}_0$ is nonzero and all $\mathbf{w}_n$ are thus nonzero as well.

Let $W = Q(\mathbf{w}_0)>0$ and define the one-sheeted hyperboloid $\mathcal{H}(W)$ as in \eqref{Hyperboloid}.
Then \eqref{WnSequence} shows $\mathbf{w}_n\in\mathcal{H}(W)$ for all $n\ge0$.
In addition, thanks to Proposition~\ref{LatticeTheorem}, $\{ \mathbf{w}_n\}_{n=0}^{\infty}$ is a discrete subset of $\mathcal{H}(W)$. 

We claim further that $\{ \mathbf{w}_n \}_{n=0}^{\infty}$ is a \emph{bounded} subset of $\mathcal{H}(W)$, consequently a finite set. 
This will prove that the periodicity is a necessary condition for the quadratic irrationality, because then at least two $\mathbf{v}_n$'s, say, $\mathbf{v}_{j_1}$ and $\mathbf{v}_{j_2}$ with $j_1 < j_2$ represent the same point in $\QQQ$, therefore
\[
	(\alpha, \beta) = [d_1, \dots, d_{j_1 - 1},
	\overline{d_{j_1}, d_{j_1 + 1}, \dots, d_{j_2}}]_{\QQQ}.
\]
We now proceed to proving the claim that $\{ \mathbf{w}_n\}_{n=0}^{\infty}$ is bounded.

It will be convenient to denote by $x_3(\mathbf{w})$ the $x_3$-coordinate of $\mathbf{w}$. 
The boundedness of $\{ \mathbf{w}_n\}_{n=0}^{\infty}$ is proven by
\begin{description}
	\item[Step 1] There exists a (large) constant $R>0$ which depends only on $D$ and $W$ such that, if
		\begin{equation}\label{WnCondition}
			|x_3(\mathbf{w}_n)| > R
\end{equation}
then
\begin{equation}\label{InductiveStep}
	|x_3(\mathbf{w}_{n+1})| < |x_3(\mathbf{w}_n)|.
\end{equation}
In other words, if $\mathbf{w}_n$ lies outside of a certain bounded set then $\mathbf{w}_{n+1}$ is closer to the origin (see \eqref{WnSequence}.)
	\item[Step 2] If 
		\begin{equation}\label{WnConditionII}
			|x_3(\mathbf{w}_n)| \le R
\end{equation}
then 
\begin{equation}\label{FinalStep}
	|x_3(\mathbf{w}_{n+1})| \le 7\sqrt2 R.
\end{equation}
In other words, if $\mathbf{w}_n$ lies \emph{inside} of the same bounded set then $\mathbf{w}_{n+1}$ cannot escape too far away.
\end{description}
It is easy to deduce the boundedness from Steps 1 and 2 above.
If $\mathbf{w}_n$ satisfies \eqref{WnCondition} for some $n$, 
then, the discreteness of $\{ \mathbf{w}_n \}_{n=0}^{\infty}$ shows that, applying Step 1 finitely many times, we obtain $|x_3(\mathbf{w}_{n+k})| \le R$ for some $k$. 
Then we apply Step 2 and then Step 1 repeatedly, to see that the sequence 
\[
	\mathbf{w}_{n+k}, 
	\mathbf{w}_{n+k+1}, 
	\mathbf{w}_{n+k+2},  \dots
\]
must lie in a bounded set $\{|x_3(\mathbf{w})| \le 7\sqrt2 R\}.$ This shows that $\{ \mathbf{w}_n\}_{n=0}^{\infty}$ is bounded.

To prove Step 1, we begin by noting that any point $(w_1, w_2, w_3)$ in the hyperboloid $\mathcal{H}(W)$ must lie close to $C_Q$ if $|w_3|$ is large. 
Indeed, the distance between the sections (the two concentric circles in Figure~\ref{MainPictureI}) of $C_Q$ and $\mathcal{H}(W)$ by the plane $x_3=w_3$ is
\[
\sqrt{w_3^2 + W} - |w_3| =
\frac{W}{\sqrt{w_3^2 + W} + |w_3|},
\]
which can be made arbitrarily small when $|w_3|$ is large.
In particular, we choose $R > 0$ large enough so that 
\begin{itemize}
	\item[($R1$)] $R > |c_0|\sqrt D/2$, and
	\item[($R2$)] $\sqrt{R^2 + W} - R < \sqrt D/2$. 
\end{itemize}
Note from \eqref{WnSequence} that
\begin{equation}\label{LastTrickWn}
	\mathbf{w}_{n+1} = \pm M_{d_{n+1}}^{-1} \mathbf{w}_n = \pm U_{d_{n+1}}H\mathbf{w}_n,
\end{equation}
and assume that $\mathbf{w}_n$ satisfies \eqref{WnCondition} above.
Then \eqref{InductiveStep} would follow from Proposition~\ref{EmilyLemma} if we show that $\mathbf{w}_n$ meets the condition \eqref{ConditionDecrease} (where $\mathbf{w}_n$ will play the role of $\mathbf{v}$ in Proposition~\ref{EmilyLemma}.)

We know from Proposition~\ref{LatticeTheorem} that
there exist $a_n, b_n, c_n \in \mathbb{Z}$ such that $\mathbf{w}_n = (\sqrt D/2)(a_n, b_n, c_n)$ for each $n\ge0$.
Let us treat the case $c_n>0$ first, as the other case is similar. 
Consider the section of $\mathcal{H}(W)$ by $x_3 = c_n\sqrt D/2$ as in Figure~\ref{MainPicture}. 
In order to verify \eqref{ConditionDecrease} for $\mathbf{w}_n$, we invoke 
Proposition~\ref{TwistedCrossProductIsClose} which says in this case that $\mathbf{w}_n$ must lie in either
\begin{enumerate}[(i)]
	\item $\{ (w_1, w_2, w_3)\in \mathcal{H}(W) \mid w_3 = c_n\sqrt D/2, 0<w_1, w_2\le c_n\sqrt D/2\}$ (the thick curve in Figure~\ref{MainPicture}), or
	\item $\{ (w_1, w_2, w_3)\in \mathcal{H}(W) \mid w_3 = c_n\sqrt D/2, 
			w_1 > c_n\sqrt D/2
			\text{ or }
			w_2 > c_n\sqrt D/2
		\}$ (the dashed curve in Figure~\ref{MainPicture}.)
\end{enumerate}
\begin{figure}
\begin{center}
\begin{tikzpicture}
\draw[->] (-3,0)--(5,0) node[right]{$x_1$};
\draw[->] (0,-3)--(0,3) node[above]{$x_2$};
\node at (4, 2.5) {$x_3 = c_n\sqrt D/2$};
\node[below left] at (0, 0) {$O$};

\draw[dotted] (0, 0) circle (2);
\draw[dotted] (0, 0) circle (2.2);

\draw[dashed] (-2.2, 2) -- (2.2, 2);
\draw[dashed] (2, -2.2) -- (2, 2.2);

\draw[thick, dashed] ([shift=(-23:2.2)]0, 0) arc (-23:116:2.2);
\draw[ultra thick] ([shift=(24:2.2)]0, 0) arc (24:65:2.2);

\draw[->] (-2.2, 1) -- (-2.2, 2);
\draw[->] (-2.2, 3.2) -- (-2.2, 2.2);
\node[right] at (-2.2, 2.2) {$<\sqrt D/2$};
\node[below left] at (2, 0) {$c_n\sqrt D/2$};
\node[below right] at (2.2, 0) {$\sqrt{c_n^2D/4 + W}$};
\end{tikzpicture}
\end{center}
\caption{The sections of $C_Q$ (the inner circle) and of $\mathcal{H}(W)$ (the outer one) by $x_3 = c_n\sqrt D/2$}
\label{MainPicture}
\end{figure}
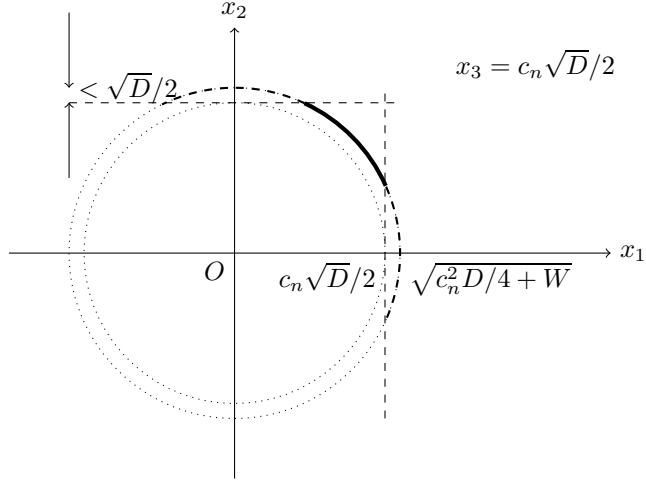
If $\mathbf{w}_n$ lies in (i) then \eqref{ConditionDecrease} is easily verified.
We show that $\mathbf{w}_n$ cannot lie in (ii), which will then complete the proof of Step 1.
From \eqref{WnCondition} and the condition ($R2$), we see that
\[
	\sqrt{\left(\frac{\sqrt D c_n}2\right)^2 + W} - \frac{\sqrt D c_n}2 < \sqrt{R^2 + W} - R < \sqrt D/2.
\]
So, if $a_n\sqrt D/2 > c_n \sqrt D/2$ then
\[
	c_n \sqrt D/2 < a_n \sqrt D/2 < (c_n + 1) \sqrt D/2,
\]
which is impossible because $a_n$ and $c_n$ are integers. 
Likewise it is impossible to have $b_n\sqrt D/2 > c_n \sqrt D/2$. 
This proves that $\mathbf{w}_n$ cannot lie in (ii).

We now move on to proving Step 2. Writing $\mathbf{w}_n = (w_1, w_2, w_3)$, the condition \eqref{WnConditionII} implies
\[
	w_1^2 + w_2^2 = w_3^2 + R \le R^2+ R,
\]
so that
\[
	|w_1| \le \sqrt{R^2 + R} \le \sqrt2R,
\]
and similarly, $|w_2| \le \sqrt2R.$
Now, we use \eqref{LastTrickWn}, the triangle inequality and the simple calculation
\[
	H \begin{pmatrix} w_1 \\ w_2 \\ w_3 \end{pmatrix}
	=
\begin{pmatrix} *  \\ * \\ -2w_1 - 2w_2 + 3w_3 \end{pmatrix}
\]
together to deduce \eqref{FinalStep}.
This completes Step 2, thus the proof of the claim that $\{\mathbf{w}_n\}$ is bounded.

\section{Some open problems}\label{SecOpen}
\subsection{Berggren graphs over real quadratic fields}
Fix a (real) quadratic extension $K$ of $\mathbb{Q}$.
Can we find Berggren trees defined over $K$, similar to Figure~\ref{PythagoreanTree}, namely, a collection of triples $(x, y, z)$ with $x^2 + y^2 = z^2$ where $x, y, z$ are positive (coprime) $K$-integers with the edges showing the actions of $M_1, M_2. M_3$?

\begin{figure}
\begin{center}
\tikzset{doublearrow/.pic={
    \draw (-0.1, 0) -- (-0.2, 0.5) node[above]{$\vdots$};
    \draw (0.1, 0) -- (0.2, 0.5) node[above]{$\vdots$};
  }
}
\tikzset{doubleRarrow/.pic={
    \draw (-0.1, 0) -- (-0.2, 0.5) node[right]{$\cdots$};
    \draw (0.1, 0) -- (0.2, 0.5) node[right]{$\cdots$};
  }
}
\tikzset{doubleLarrow/.pic={
    \draw (-0.1, 0) -- (-0.2, 0.5) node[left]{$\cdots$};
    \draw (0.1, 0) -- (0.2, 0.5) node[left]{$\cdots$};
  }
}
\begin{tikzpicture}[->, >=latex', auto, scale=2]
%\draw[help lines,step=.5] (-3, 0) grid (3, 3);
\node (classk-2) at (-1, 0) {};
\node (classk-1) at (-1.5, 1) {$[x_{k-1}, y_{k-1}, z_{k-1}]$};
\node (classk) at (-1, 2) {$[x_{k}, y_{k}, z_{k}]$};
\node (class1) at (1, 2) {$[x_{1}, y_{1}, z_{1}]$};
\node (class2) at (1.5, 1) {$[x_{2}, y_{2}, z_{2}]$};
\node (class3) at (1, 0) {};
\node (base) at (0, 0) {$\cdots$};

\draw (classk-1) to node[font=\footnotesize]{${M_{d_{k-2}}}$} (classk-2);
\draw (classk) to node[font=\footnotesize]{${M_{d_{k-1}}}$} (classk-1);
\draw (class1) to node[font=\footnotesize]{${M_{d_k}}$} (classk);
\draw (class2) to node[font=\footnotesize]{${M_{d_1}}$} (class1);
\draw (class3) to node[font=\footnotesize]{${M_{d_2}}$} (class2);

\pic at (-1, 2.1) {doublearrow};
\pic at (1, 2.1) {doublearrow};
\pic at (1.9, 1) [rotate=-90] {doubleRarrow};
\pic at (-2.3, 1) [rotate=90] {doubleLarrow};
\end{tikzpicture}
\end{center}
\caption{Purely periodic points form a circular root. \label{CircularRoot}}
\end{figure}
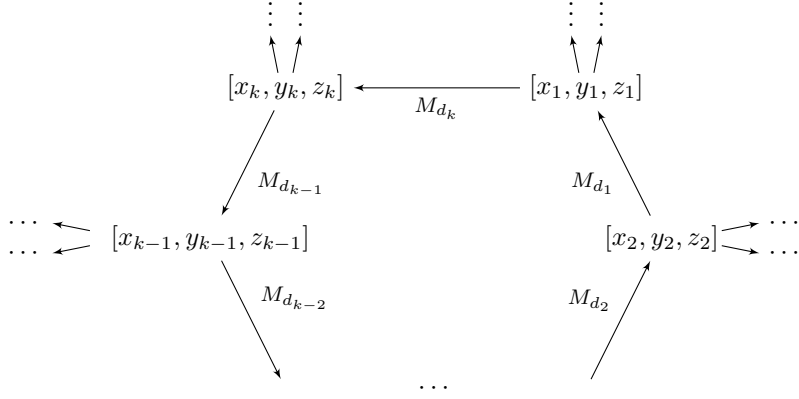
When we consider $K$-integer triples $(x, y, z)$, it is convenient to remove the ambiguity stemming from the existence of \emph{units} in $K$.
As is well-known, there exists a \emph{fundamental unit}, say, $\epsilon_K > 1$, such that the set of all the units of $K$ are given by $\{ \pm {\epsilon_K}^{\mathbb{Z}}\}$.
(See, for example, Proposition 13.1.6 in \cite{IR90}.) Let us denote by $[x, y, z]$ the equivalent class defined by the relation
\[
(x, y, z) \sim (x', y', z') \Longleftrightarrow (x, y, z) = \epsilon_K^k (x', y', z')
\]
for some $k\in \mathbb{Z}$. 
Note that any triple in the class $[x, y, z]$ represents a unique point $(x/z, y/z) \in\QQQ$. Write its Romik digit expansion
\[
(x/z, y/z) = [d_1', \dots, d_t', \overline{d_1, \dots, d_k}]_{\QQQ}.
\]
Then, our results in \S\ref{SecProof} show that the class $[x, y, z]$ belongs to a unique Berggren tree that grows out of a \emph{circular root}, as pictured in Figure~\ref{CircularRoot}.
Here the points $(x_1/z_1, y_1/z_1), \dots, (x_k/z_k, y_k/z_k)$ have the Romik digit expansions
\begin{align*}
(x_1/z_1, y_1/z_1) &= [\overline{d_1, \dots, d_k}]_{\QQQ}, \\
(x_2/z_2, y_2/z_2) &= [\overline{d_2, d_3, \dots, d_k, d_1}]_{\QQQ}, \\
\dots & \\
(x_k/z_k, y_k/z_k) &= [\overline{d_k, d_1, d_2,  \dots, d_{k-1}}]_{\QQQ}.
\end{align*}
In other words, a tree growing out of the circular root will contain all the triples whose Romik digit expansions share the same period, up to circular permutation. 
Perhaps, a better nomenclature in this situation is a \emph{Berggren graph}, as it is no longer a tree in the graph-theoretic sense.
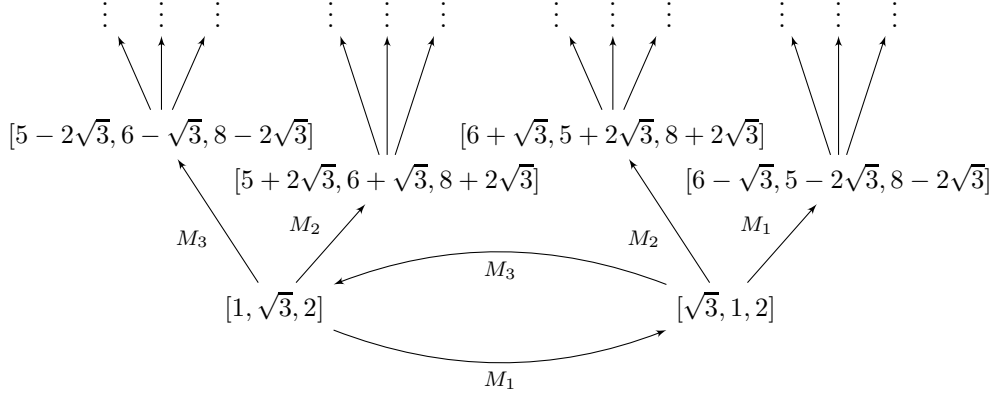
\begin{figure}
%\begin{figure}
\begin{center}
\begin{tikzpicture}[->, >=latex', auto, xscale=1.5]
\node (base1) at (2, 0) {$[1, \sqrt 3, 2]$};
\node (base2) at (6, 0) {$[\sqrt3, 1, 2]$};

\node (1f1) at (1, 2.3) {$[5 - 2\sqrt3, 6 - \sqrt 3, 8 - 2\sqrt3]$};
\node (1f2) at (3, 1.7) {$[5 + 2\sqrt3, 6 + \sqrt 3, 8 + 2\sqrt3]$};
\node (1f3) at (5, 2.3) {$[6 + \sqrt 3, 5 + 2\sqrt3, 8 + 2\sqrt3]$};
\node (1f4) at (7, 1.7) {$[6 - \sqrt 3, 5 - 2\sqrt3, 8 - 2\sqrt3]$};

\node (21f1) at (0.5, 4) {$\vdots$};
\node (22f1) at (1.0, 4) {$\vdots$};
\node (23f1) at (1.5, 4) {$\vdots$};

\node (21f2) at (2.5, 4) {$\vdots$};
\node (22f2) at (3.0, 4) {$\vdots$};
\node (23f2) at (3.5, 4) {$\vdots$};

\node (21f3) at (4.5, 4) {$\vdots$};
\node (22f3) at (5.0, 4) {$\vdots$};
\node (23f3) at (5.5, 4) {$\vdots$};

\node (21f4) at (6.5, 4) {$\vdots$};
\node (22f4) at (7.0, 4) {$\vdots$};
\node (23f4) at (7.5, 4) {$\vdots$};

\draw (base1) [bend right] to node[below, font=\footnotesize]{${M_1}$} (base2);
\draw (base2) [bend right] to node[font=\footnotesize] {${M_3}$} (base1);

\draw (base1) to node[font=\footnotesize]{${M_3}$} (1f1);
\draw (base1) to node[font=\footnotesize]{${M_2}$} (1f2);

\draw (base2) to node[font=\footnotesize]{${M_2}$} (1f3);
\draw (base2) to node[font=\footnotesize]{${M_1}$} (1f4);

\draw (1f1) to (21f1); 
\draw (1f1) to (22f1); 
\draw (1f1) to (23f1); 

\draw (1f2) to (21f2); 
\draw (1f2) to (22f2); 
\draw (1f2) to (23f2); 

\draw (1f3) to (21f3); 
\draw (1f3) to (22f3); 
\draw (1f3) to (23f3); 

\draw (1f4) to (21f4); 
\draw (1f4) to (22f4); 
\draw (1f4) to (23f4); 
\end{tikzpicture}
\end{center}
\caption{An example of Berggren graph in $\mathbb{Q}(\sqrt 3)$, rooted at $[1, \sqrt3, 2]$ and $[\sqrt3, 1, 2]$. \label{BerggrenGraph}}
\end{figure}
An example of a Berggren graph in $\mathbb{Q}(\sqrt3)$ rooted at $[1, \sqrt3, 2]$ and $[\sqrt3, 1, 2]$ is provided in Figure~\ref{BerggrenGraph}.

The upshot of our discussion so far is the following. In order to answer the question posted at the beginning of this section, we must find all the circular roots defined over $K$, from which we then grow Berggren graphs to obtain the entire integer $K$-triples.
Therefore, a natural question to ask is, \textit{does there exist an algorithm for finding all the purely periodic points in $\QQQ$ defined over $K$?} 
Relatedly, here is another question, perhaps a bit weaker. For a finite Romik digit sequence $\vec d = [d_1, \dots, d_k]$ with $d_j \in \{1, 2, 3 \}$, define $(x(\vec d), y(\vec d)) \in \QQQ$ to be the point whose Romik digit expansion is purely periodic with $\vec d$ being its period, namely,
\[
(x(\vec d), y(\vec d)) = [\overline{d_1, \dots, d_k}]_{\QQQ}.
\]
Then, does the number
\[
N(k, K) := \# \{ \vec d \in \{ 1, 2, 3 \}^k \mid x(\vec d), y(\vec d) \in K \}
\]
of such points \emph{that are defined over} $K$ tend to infinity as $k\to\infty$?
If so, how fast? Does the growth rate depend on (the discriminant of) $K$?

\subsection{Romik systems defined by other quadratic forms}
Another way to generalize the work in this paper is to consider Romik systems stemming from other quadratic forms considered in \cite{CNT}.
We believe that, for certain forms such as $Q(x, y, z) = x^2 + xy + y^2 - z^2$, much of the arguments presented here will carry over essentially in the same way.
But, we don't expect that every quadratic form (with a Berggren tree) would always yield the same result.
Particularly interesting to us is the form $Q(x, y, z, w) = x^2 + y^2 + z^2 - w^2$ for \emph{Pythagorean quadruples}.
This quadratic form and its Berggren tree have been recently studied by Chaubey et al in \cite{CFHS17} in a slightly different context.
Will a Romik system arising from its Berggren trees consisting of Pythagorean quadruples have similar properties as described in this paper? Will there be any interesting arithmetical consequences?

%\begin{acknowledgment}{Acknowledgment.}
%The first author is grateful to his former and current undergraduate students from Muhlenberg College, with whom he has had fruitful and inspiring collaborative relationship over the past years.
%The second author was supported by National Research Foundation of Korea (NRF-2015R1A2A2A01007090).
%\end{acknowledgment}

\end{document}